\documentclass{article}
\usepackage{amsmath, amsthm, amscd, amsfonts, amssymb, graphicx, color,mathrsfs,mathtools,enumerate}
\usepackage[mathcal]{euscript}
\usepackage{yfonts}
\usepackage{lineno}
\usepackage{authblk}
\usepackage{dsfont}
\usepackage{bm}
\usepackage{bbm}
\usepackage{musicography}
\usepackage[bookmarksnumbered, colorlinks, plainpages]{hyperref}
\hypersetup{colorlinks=true,linkcolor=red, anchorcolor=green, citecolor=cyan, urlcolor=red, filecolor=magenta, pdftoolbar=true} 
\modulolinenumbers[5]
\usepackage{xcolor}
\usepackage{tikz}

\newtheorem{theorem}{Theorem}[section]
\newtheorem{lemma}[theorem]{Lemma}
\newtheorem{proposition}[theorem]{Proposition}
\newtheorem{corollary}[theorem]{Corollary}
\newtheorem{remark}[theorem]{Remark}
\newtheorem{definition}[theorem]{Definition}

\theoremstyle{definition}

\theoremstyle{remark}
\newtheorem{rmk}[theorem]{Remark}

\numberwithin{equation}{section}

\newcommand{\N}{\mathbb N}
\newcommand{\R}{\mathbb R}
\newcommand{\norm}[1]{\left\lVert#1\right\rVert}
\newcommand{\abs}[1]{\left\lvert#1\right\rvert}
\newcommand{\ps}[2]{\left\langle#1,#2\right\rangle}
\newcommand{\Rn}{\mathbb{R}^n}
\newcommand{\Sn}{\mathbb{S}^{n-1}}
\newcommand{\sm}{\smallsetminus}
\newcommand{\mr}{\mathbin{\vrule height 1.6ex depth 0pt width 0.13ex\vrule height 0.13ex depth 0pt width 1.3ex}}

\allowdisplaybreaks

\title{\textbf{Heat content asymptotics for sets with positive reach}}
\author[$\dagger$]{Paolo De Fazio}
\author[$\star$]{Michele Miranda Jr.}
\date{}

\affil[$\dagger$]{\small{Dipartimento di Scienza e Alta Tecnologia, Universit\`a degli studi dell'Insubria,
Via Valleggio, 11, I-22100 Como,
 email-address: paolo.defazio@uninsubria.it}}
\affil[$\star$]{\small{corresponding author, Dipartimento di Matematica e Informatica, Universit\`a degli studi di Ferrara,
via Nicol\`o Machiavelli 30, I-44121 Ferrara, 
email-address: michele.miranda@unife.it}}

\providecommand{\keywords}[1]{{\textit{Keywords}:} #1}
\providecommand{\subjclass}[1]{{\textit{2020 Mathematics Subject Classification}:} #1}

\begin{document}

\maketitle

\begin{center}
\textit{In memory of Umberto Massari.}
\end{center}

\begin{abstract}
In this paper we study the heat content for sets with 
positive reach.
In details, we investigate the asymptotic behavior of the 
heat content 
of bounded subsets of the Euclidean space with positive 
reach. 
The concept of positive reach was introduced by 
Federer in \cite{fed_1959} and widely developed in the 
following years (see for instance
the recent book by Rataj and Zh{\"a}le 
\cite{rat_zah_2019}). 
It extends the class of sets with 
smooth boundaries to include certain non-smooth and 
singular sets while 
still admitting a well-defined normal geometry. For such 
sets $E\subseteq\Rn$, we analyze the short-time 
asymptotics of the heat 
content $\|T_t\mathbbm{1}_E\|_2$, where 
$T_t\mathbbm{1}_E$ is 
the soluzion of the heat 
equation in $\Rn$ with initial condition $\mathbbm{1}_E$. 
The present paper is in the spirit of 
Angiuli, Massari and Miranda Jr.\cite{ang_mas_mir_2013}, 
but the technique's used here are completely different
and also the final result is slightly different.
\noindent
\\\phantom{a}\\
\\
\end{abstract}

\section{Introduction}

Characterizing the class of sets with finite perimeter is 
a topic widely 
studied in Geometric Measure Theory. 
Working in this direction, Ledoux \cite{led_1994}
introduced the functional 
\begin{align}
K_t(E,F)&=\int_F T_t(\mathbbm{1}_E)(x)\,dx,\ \ t\geq0,
\end{align}
where $E,F$ are Borel set of $\Rn$, $\mathbbm{1}_E$ the 
characteristic function of the set $E$ and $T_t$ is the 
Heat semigroup 
in $\Rn$ that is 
given by
\begin{align}
    &T_0f(x)=f(x),\ \ x\in\R^n,\\
    &T_t f(x)=
    \frac{1}{(4\pi t)^{\frac{n}{2}}}
    \int_E f(y)\,e^{-\frac{\norm{x-y}^2}{4t}}\,dy,\ \ t>0,\ x\in\Rn,
    \end{align}
where $f$ is a bounded Borel function.
He proved that 
\begin{align}
\lim_{t\rightarrow 0^+}\sqrt{\frac{\pi}{t}} 
K_t(B,B^c)=P(B)=
\mathcal{H}^{n-1}(\partial B), 
\end{align}
with $P(B)=\mathcal{H}^{n-1}(\partial B)$ being 
the perimeter, i.e. the $(n-1)$-dimensional surface
measre of $\partial B$ of any ball $B$ of 
$\R^n$,  and $B^c=\R^n\sm B$. Moreover, for any set 
$E\subseteq\Rn$ with regular boundary, he proved that
\[
\sqrt{\frac{\pi}{t}} K_t(E,E^c)\leq {\mathcal H}^{n-1}
(\partial E),\ \ t>0,
\]
where $E^c=\R^n\sm E$.
For all $t\geq 0$ we set
\begin{align}
&\norm{T_t\mathbbm{1}_E}_{L^2(\R^n)}=\int_E
T_{2t}\mathbbm{1}_E(x)\,dx=\mathcal{L}^n(E)-K_{2t}
(E,E^c),\label{norma2}\\
&\norm{T_t\mathbbm{1}_E-\mathbbm{1}_E}_{L^1(\R^n)}=2K_{t}
(E,E^c),\label{norma1}
\end{align}
where $\mathcal{L}^n(E)$ is the $n$-dimensional Lebesgue 
measure of $E$.
We refer to \cite[Remark 3.5]{mir_pal_par_pre_2007} for a 
proof of \eqref{norma1}.

From these considerations, Ledoux deduced an equivalent 
proof of the isoperimetric property of balls. 
More recently in \cite{mir_pal_par_pre_2007}, the authors 
showed that 
a set $E\subseteq\R^n$ has finite perimeter if 
and only if 
\begin{align}
\liminf_{t\rightarrow 0^+}\frac{K_t(E,E^c)}{\sqrt{t}}
<+\infty,
\end{align}
and in this case 
\begin{align}
P(E)=\lim_{t\rightarrow 0^+}
\sqrt{\frac{\pi}{t}}K_t(E,E^c).
\end{align}
We refer also to \cite{ang_mir_pal_par_2009} for further 
development
involving more general operators in divergence form on 
domains with Dirichlet boundary conditions, rather 
then the Laplace operator on $\Rn$ and to 
\cite{mar_mir_sha_2016} for the generalization in 
metric spaces. We also point out that the relation 
between perimeter and heat semigroup was estabilished
by De~Giorgi \cite{deg_1953} (and subsequently works 
\cite{deg_1954,deg_1955}) where the definition of sets 
with finite perimeter and functions with bounded variation
in $\Rn$ was introduced.

The object of our investigation is the Taylor 
expansion in time of the function
\begin{equation}
      f_E(t)=\int_{E^c}T_{t^2}(\varphi\mathbbm{1}_E)(x)\,dx=
      \frac{1}{(4\pi)^{\frac{n}{2}}t^n}\int_{E^c}\int_E 
      e^{-\frac{\norm{x-y}^2}{4t^2}}\varphi(y)\,dy\,dx,\ \ \ \varphi\in C^1_b(\Rn),\ t>0,
\end{equation}
that is called \textit{heat content of $E$}.
Of course $f_E(0)=0$ for 
all $\varphi\in C^1_b(\R^n)$ and $f_E(t)$ is equal to 
$K_{t^2}(E,E^c)$ if we choose $\varphi\equiv 1$. 

The main achievement of this paper is to compute a second 
order Taylor expansion of $f_E(t)$ assuming that $E$ is 
a compact with positive reach (see Theorem 
\ref{teorema_sviluppo_ps}). As a Taylor expansion of 
the first order of $K_{t^2}(E,E^c)$ characterizes the 
class of sets of finite perimeter, our result is one of 
first steps to characterize a class of non-smooth and 
maybe singular sets with finite perimeter and a weak (but 
good enough) notion of curvatures. 
In \cite{ang_mas_mir_2013} the authors proved a Taylor 
expansion of third order of the mapping $K_{t^2}(E,E^c)$ 
assuming that $E$ has a $C^{1,1}$-skeleton. 
In the smooth case, similar problems has been deeply 
studied for instance in 
\cite{sch_2021,van_2013,van_1989,gil_van_2012,
gil_van_2015bis,gil_kan_van_2013,gil_van_1999,
gil_van_2004,van_2018,van_gil_see_2008,van_git_2015,
van_leg_1994}, where
the semigroups considered are the Dirichlet or Neumann 
semigroup on a suitable domain or the semigroup 
associated to the Laplace-Beltrami operator on a smooth 
manifold. 
Without assuming any regularity, we have some 
relevant results using the Dirichlet heat semigroup 
in \cite{van_1998} for $n=2$ and in \cite{van_1994} some 
lower and upper bounds are proven in any dimension. 
We stress that assuming $E$ to have positive reach
is a regularity assumption on the boundary of $E$ that,
at best of our knowledge, it is the first result in this
direction in $\R^n$ with $n>2$ (see Theorem \ref{teorema_sviluppo_ps}).
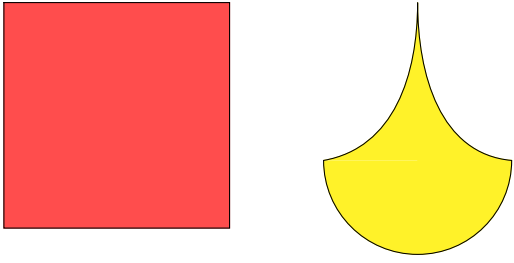
\begin{figure}[htbp]
  \centering
  \mbox{%
    \begin{minipage}[l]{.5\textwidth}
      \begin{tikzpicture}[scale=1,cap=round]
 \tikzset{axes/.style={}}
\begin{scope}[style=axes]
\filldraw[yellow! 90! white ](0.25,0.4) to [out=10,in=-90]  (1.5,2.5)
 to [out=-90, in=175] (2.75,.4);
 \filldraw[yellow! 90! white ] (0.25,0.4) arc[start angle=180, end angle=360, radius=1.25]--(1.5,0.4);
  \draw (0.25,0.4) arc[start angle=180, end angle=360, radius=1.25];
  \draw(0.25,0.4) to [out=10,in=-90]  (1.5,2.5)
 to [out=-90, in=175] (2.75,0.4); 
 \filldraw[red! 70! white ] (-4,2.5)--(-4,-0.5)--(-1,-0.5)--(-1,2.5)--(-4,2.5);
\draw (-4,2.5)--(-4,-0.5)--(-1,-0.5)--(-1,2.5)--(-4,2.5);
\end{scope}
\end{tikzpicture}
    \end{minipage}%
    \qquad
    \begin{minipage}[r]{.4\textwidth}
      \caption{{\small Examples of singular sets with positive reach.}}
    \end{minipage}
   }
\end{figure}
We point out that sets with positive reach are not 
enough to characterize the class of sets such that
$f_E(t)$ admits a second order Taylor expansion (see 
Remark \ref{remark_finale} for more details).

Sets with positive reach represent one of the most 
natural environments to try these kind of computations 
since curvatures are defined in a weak sense at the 
boundary of these sets. They were introduced in the '50s 
by Federer who was looking for a wider class of sets for 
which a Steiner-type Formula holds. Let 
$E\subseteq\mathbb{R}^n$, let $r>0$ and let $[E]_r$ be 
the $r$-
parallel set of $E$, namely the set consisting of all 
points $x\in\mathbb{R}^n$ whose distance from $E$ is less 
then $r$. An interesting formula to compute the $n$-
dimensional measure of $[E]_r$ was given by Jacob Steiner 
for polytopes and smooth convex bodies in 
dimension $2$ and $3$. Then Weyl in the '30s proved in 
\cite{wey_1939} a Steiner-type Formula for $C^2$ compact 
manifolds embedded in $\mathbb{R}^n$. In \cite{fed_1959} 
sets with positive reach were introduced as sets for 
which a projection map 
is defined on ${\rm Unp}(E)$, namely the set of all 
points having unique 
projection on $E$. More recent developments are due to 
Rataj and Z\"{a}hle 
in \cite{rat_zaj_2017,rot_zah_1990,zah_1986} who made 
this topic easier to understand using a more modern 
language and presenting simpler proofs. For a more 
systematic treatment, we refer to \cite{rat_zah_2019} and 
to the references therein. We also remember 
\cite{amb_col_vil_2008,hug_1998,hug_gun_wei_2004,men_san_2019,san_2020,stacho_1979}.

The paper is organized as follows.

In Section \ref{notations} we establish our notation and 
known results on rectifiability, sets with finite 
perimeter and sets with positive reach. Let 
$E\subseteq\R^n$ be a set with positive reach. Following 
the approach of \cite{rat_zah_2019,zah_1986}, we set 
\[
 {\rm nor}^\star(E)=\{(x,\nu)\in\mathbb{R}^n\times\mathbb{S}^{n-1} \ \mbox{s.t.}\ x\in\partial E\ \mbox{and}\ \nu\in{\rm Nor}(E,x)\},
 \]
 where the normal cone ${\rm Nor}(E,x)\}$ is defined in 
 Definition \ref{coninormalietangenti}. Then 
 $\mathcal{H}^{n-1}$-almost every $(x,\nu)\in
 {\rm nor}^\star(E)$ the generalized curvatures are well 
 defined, and the following Steiner-type formula holds 
 true
\begin{equation}
\mathcal{L}^n([E]_r\sm  E)=\sum_{k=0}^{n-1}
\frac{r^{n-k}}{n-k}\int_{{\rm nor}^\star(E)}
\sigma_E^{n-1-k}(x,\nu)d\,\mathcal{H}^{n-1}(x,\nu),
\end{equation}
 where $\sigma_E^{k}(x,\nu)$ denotes the generalized 
 symmetric functions of the principal curvatures of order 
 $k$ of $E$ at $(x,\nu)\in{\rm nor}^\star(E)$ (see 
 \eqref{sigma0}, \eqref{sigmak}). 
 
In Lemma \ref{decoE}, we point out the features of the 
sets
\begin{align}
\Sigma_{k}=\Big\{x\in E\ \big|\  {\rm dim}({\rm Nor}
(E,x))= n-k
\Big\},\ \ k=1,...,n,
\end{align}
are such that 
\begin{align}
\partial E=\bigcup_{k=1}^n\Sigma_k,
\end{align}
and for $\mathcal{H}^{n-1}$a.e. $z\in \Sigma^k$ exactly  
$n-1-k$ of the generalized curvatures are equal 
to $+\infty$.

In Section \ref{ps_rettificabilità} we recall some 
results of \cite{amb_col_vil_2008} where they proved that 
a set $E\subseteq\Rn$ with positive reach has finite 
perimeter and 
\[
P(E)=\mathcal{H}^{n-1}(\mathcal FE)\leq\mathcal{H}^{n-1}
(\partial E)<+\infty,
\]
where $\mathcal FE$ is the reduced boundary of $E$ and 
$\partial E$ is the topological boundary of $E$. From 
that we characterized the blow-up at each point $x_0$ 
belonging to $\partial E$. More precisely, in Theorem 
\ref{convhaus} we proved that for all $r>0$ the sets 
$\displaystyle{\frac{E-x_0}{\rho}\cap \overline{B_r(0)}}$ 
converge with respect to the Hausdorff distance as 
$\rho\to0^+$ to ${\rm Tan}(E,x_0)\cap \overline{B_r(0)}$, 
where ${\rm Tan}(E,x_0)$ is the tangent cone defined in 
Definition \ref{coninormalietangenti}. By Proposition 
\ref{beer}, the sets $\displaystyle{\frac{E-x_0}{\rho}}$ 
locally converge in measure as $\rho\to0^+$ to 
${\rm Tan}(E,x_0)$ (see Corollary \ref{blowup}).

In Section \ref{taylor}, we proved the second order 
Taylor expansion for the Heat content for a compact set 
with positive reach $E\subseteq\Rn$ (see Theorem 
\ref{teorema_sviluppo_ps}). More precisely, we proved that
\begin{align}\label{sviluppops_intro}
f_E(t)&=
\frac{t}{\sqrt{\pi}}
\int_{\mathcal{F}E}\varphi(x)d\mathcal{H}^{n-1}(x)
+
t^2\biggl[\int_{\mathcal{F}E}\left(\alpha_nH_E(x)
\varphi(x)-\ps{\nabla\varphi(x)}
{\nu_E(x)}\right)\,d\mathcal{H}^{n-1}(x)\nonumber\\
&+\int_{\Sigma_{n-2}} c_{n-2}(
x)\varphi(x)\,d\mathcal{H}^{n-2}(x)\biggr]+o(t^2),
\end{align}
where 
\begin{align}
&\alpha_n=2^{n-2} \pi^{\frac{n-1}{2}}(n-1),\\
&H_E(x)=\frac{1}{n-1}\sum_{j=1}^{n-1}k_j(x),\\
&\Sigma_{n-2}=\Bigl\{x\in E\ \big|\  {\rm dim}({\rm Nor}
(E,x))=2\Bigr\},\\
&c_{n-2}(x)=\frac{1}{(4\pi)^{\frac{n}{2}}}\int_{{\rm Nor}
(E,x)}d\mathcal{H}^{2}(y)\int_{{\rm Tan}
(E,x)}\norm{y}\,e^{-\frac{\norm{w-y}^2}{4}}\,dw,\ \ 
x\in\Sigma_{n-2},
\end{align}
and the curvatures $k_i(x)$ are meant as in Remark 
\ref{casoregolare} since $\mathcal{F}E$ is 
$\mathcal{H}^{n-1}$-rectifiable of class $C^2$.

In particular, taking $\varphi\equiv 1$ we get 
{\small\begin{align}\label{sviluppokintro}
K_{t^2}(E,E^c)&=
\frac{t}{\sqrt{\pi}}P(E)+t^2\biggl[\alpha_n
\int_{\mathcal{F}E}H_E(x)\,d\mathcal{H}^{n-1}(x)
+\int_{\Sigma_{n-2}} c_{n-2}(x)\,d\mathcal{H}^{n-2}
(x)\biggr]+o(t^2).
\end{align}}
By \eqref{norma2}, \eqref{norma1} and \eqref{sviluppokintro} we also obtain that
{\footnotesize\begin{align}
&\norm{T_t\mathbbm{1}_E}_{L^2(\R^n)}=\sqrt{\frac{2t}{\pi}}P(E)+2t\biggl[\alpha_n\int_{\mathcal{F}E}H_E(x)
\,d\mathcal{H}^{n-1}(x)
+\int_{\Sigma_{n-2}} c_{n-2}(x)\,d\mathcal{H}^{n-2}(x)\biggr]+o(t),\\
&\norm{T_t\mathbbm{1}_E-\mathbbm{1}_E}_{L^1(\R^n)}=2\sqrt{\frac{t}{\pi}}P(E)+2t\biggl[\alpha_n\int_{\mathcal{F}E}H_E(x)
\,d\mathcal{H}^{n-1}(x)
+\int_{\Sigma_{n-2}} c_{n-2}(x)\,d\mathcal{H}^{n-2}(x)\biggr]+o(t).
\end{align}}

\section{Notations and preliminary results}\label{notations}

In this section we recall results and notation of geometric measure theory and on sets with positive reach.

We work in the Euclidian space $\Rn$, with $n\geq 2$, and we denote 
by $\norm{\,\cdot\,}$ and $\ps{\cdot}{\cdot}$ the standard norm and 
the standard inner product of $\Rn$, respectively. For all open subset 
$\Omega$ of $\Rn$ we denote by $\mathcal{B}(\Omega)$  the $
\sigma$-algebra of Borel sets of $\Omega$. For all $r>0$ and 
$x_0\in \Rn$, $B_r(x_0)$ denotes the open ball of center $x_0$ and 
radius $r$ and $S_r(x_0)$ is the sphere of center $x_0$ and radius $r
$; in particular, we write $\Sn$ for $S_1(0)$. For all $s\geq 0$, $
\mathcal{H}^\alpha$ is the $\alpha$-dimensional Hausdorff measure, 
$\mathcal{L}^n$ is the Lebesgue measure on $\Rn$ and by we denote by $\omega_{n}$ is the $n$-dimensional Lebesgue measure of $B_1(0)$.
Given a subset $E$ of $\Rn$, we define the Hausdorff dimension of 
$E$ as the real number  
\[
{\rm dim}_{\rm H} (E)=\inf
\Big\{ \alpha\ \big|\  0\leq
 \alpha<+\infty, \ \mathcal{H}^\alpha(E)=0\Big\}=
 \sup\Big\{ \alpha\ \big|\ 0\leq \alpha<+\infty, 
 \ \mathcal{H}^\alpha(E)=+\infty\Big\}.
\]
Moreover $\mathring{E}$ is the interior of $E$, $\overline{E}$ is the 
closure of $E$, $E^c$ in the complement of the set $E$ in $\R^n$. The convex 
hull of $E$ is the set 
\[
co(E):=\bigcap_{S\supseteq E}S,
\] 
whenever 
$S$ is convex (i.e. $co(E)$ is the smallest convex set containing $E$). 

Given $x\in\Rn$, $\delta_E(x)$ denotes the distance of $x$ from the 
set $E$, namely we are defining the mapping $\delta_E:\Rn
\longrightarrow[0,+\infty]$ such that 
\[
\delta_E(x)=\inf_{y\in E} 
\norm{x-y},\quad x\in\Rn.
\] 
$\delta_E$ is a $1$-Lipschitz map (see e.g. \cite[Lemma 4.2]
{rat_zah_2019}).
Moreover the $s-$parallel set to $E$ is  
\[
[E]_s=\{x\in \Rn\ |\ 
\delta_E(x)\leq s\},
\]
$s-$annulus set to $E$ is 
\begin{equation}\label{annulus}
A_s(E)=[E]_s\smallsetminus\mathring{E}
\end{equation}
and the dual of $E$ is the set 
\[
\mbox{Dual}(E):=\Big\{v\in \Rn\ \big|\ps{v}{u}\leq0\ \forall u\in E
\Big\}.
\]  
Let $C\subseteq\Rn$. We say that $C$ is a positive cone or simply 
a cone if $\lambda c\in C$ whenever $\lambda\geq 0$ and $c\in C$.
Clearly $\mbox{Dual}(E)$ is a closed convex cone. 
We denote by $C(E)$ the cone generated by $E$
\[
C(E)=\{\lambda x\ \big|\ \lambda\geq0 , x\in E\}
\]
and for $x_0\in\Rn$, we set
\[
C(x_0,E)=x_0+C(E)
\]

Let $C\subseteq\Rn$ be a cone: we set 
\[
{\rm dim}(C):=\min\Big\{ {\rm dim}(V)\ \big|\  C\subseteq V, V \mbox{vector space}\Big\}
\]
Since if $C\subseteq V$ with ${\rm dim}(C)={\rm dim}(V)$, 
$V$ vector space, then $V^\perp\subseteq {\rm Dual}(V)$, 
we deduce that
\[
{\rm dim}(C)+{\rm dim}({\rm Dual}(C))\geq n,
\]
with the equality  if and only if $C$ is a vector space.

We recall the following result and give a short proof for
reader convenience.

\begin{lemma}
Let $C\subseteq\Rn$ be a closed convex cone. Then ${\rm dim}(C)={\rm dim}_H(C)$.
\end{lemma}

\begin{proof}
If ${\rm dim}(C)=0$ then $C=\{0\}$ and then the conclusion holds. 

If ${\rm dim}(C)=d\geq 1$ then there exist exactly 
$v_1,...,v_d\in \Sn\cap C$ linearly independent. 
Indeed if there exist $d+1$ independent vectors, $C$
then the smallest vector space containing $C$ has dimension at least 
$(d+1)$.
On the other hand, $C$ contains more than $d-1$ independent 
vectors. If not $C$ would be contained in $(d-1)$-dimensional 
space which is a contradiction.  
By convexity  we obtain that $co(\{v_1,...,v_d\})\subseteq C$ and 
then $C$ has nonempty interior in $V={\rm span}\{v_1,...,v_d\}$, 
i.e. there exists $\rho>0$ and $x\in C$ such that $B_\rho(x)\cap V
\subseteq C$. Hence we get
\[
 0<\mathcal{H}^d(B_\rho(x)\cap V)\leq \mathcal{H}^d(C).
 \]
 In addition, since $C\subseteq V$, $\mathcal{H}^\alpha
 (C)\leq \mathcal{H}^\alpha(V)=0$ for all $\alpha>d$;
 hence $\mbox{dim}_H(C)=d$. 
\end{proof}       
We denote by $C^k_b(\R^n)$ the space of the bounded and continuous functions $f:\R^n\longrightarrow\R$ whose derivatives up to order $k$ are bounded and continuous. Moreover, we write $C_b(\R^n)$ instead of $C^0_b(\R^n)$ and we set $\displaystyle{C^\infty_b(\R^n)=\bigcap_{k\in\N}C^k_b(\R^n)}$.
                   
\subsection{Rectifiability}
In this section we recall basic facts of rectifiability sets; the main references are \cite{amb_fus_pal_2000,anz_ser_1994,del_2003}.

\begin{definition}
Let $E$ be a Borel subset of $\Rn$.
\begin{enumerate}[{\rm (i)}] 
\item $E$ is $k$-rectifiable if  there exist a Lipschitz functions $f:\R^k\longrightarrow\mathbb{R}^n$ and a bounded subset $B$ of $\R^k$ such that $E=f(B)$.
\item $E$ is countably $k$-rectifiable if $E$ is a countably union of $k$-rectifiable sets.
\item $E$ of $\Rn$ is  countably $\mathcal{H}^k$-rectifiable if there exists a countable family of Lipschitz functions $\{f_i\}_{i\in\mathbb{N}}$, $f_i:\R^k\longrightarrow\mathbb{R}^n$ such that:
    \begin{enumerate}[(i)]
           \item $\displaystyle{E\subseteq\bigcup_{i\in\mathbb{N}}f_i(\R^k)}$,
       \item $\displaystyle{\mathcal{H}^k\Bigl(E\smallsetminus\bigcup_{i\in\mathbb{N}}f_i(\R^k)\Bigr)=0}$.
       \end{enumerate}
       \item we say that $E$ is $\mathcal{H}^k$--rectifiable
       of class $C^s$ if
       it is countably $\mathcal{H}^k$--rectifiable and  the functions 
       $f_i$, $i\in \N$,  are of class $C^s$, $s\in \N$. 
  \end{enumerate}
  \end{definition}
       
Let $\mu$ be a $\R^{m}$-valued Radon measure on an open set $\Omega\subseteq\Rn$. We say that $\mu$ is $k$-rectifiable if there exist a countably $\mathcal{H}^k$-rectifiable set $E$ and a Borel map $\theta:E\longrightarrow\R^{m}$ such that $\mu=\theta\mathcal{H}^k\mr E.$
Moreover we say that $\mu$ has a $k$-dimensional approximate tangent space $\pi$ with multiplicity $\theta\in\R^{m}$ if given $B\in\mathcal{B}(\Rn)$, $\rho^k\mu(x+\rho B)$ locally weakly$^*$ converges to the Radon measure $\theta\mathcal{H}^k\mr\pi$ in $\Rn$ as $\rho\searrow 0^+$, namely 

\[
\lim_{\rho\searrow 0^+}\frac{1}{\rho^k}\int_{\Omega}\varphi\left(\frac{y-x}{\rho}\right)\,d\mu(y)=\theta(x)\int_\pi\varphi(y)\,d\mathcal{H}^k(y),\ \ \mbox{for all}\ \varphi\in C_c(\Rn).
\] 
In this case we write 
\[
{\rm Tan}^k(\mu,x)=\theta\mathcal{H}^k\mr\pi.
\]

Let $\Omega\subseteq\Rn$ be open. For all $E\in\mathcal{B}(\Rn)$, we say that $E$ has finite perimeter in $\Omega$ if there exists a $\Rn$-valued Radon measure in $\Omega$, denoted by, $D\mathbbm{1}_E$, such that for all $\phi\in C^\infty_c(\Omega)$ we have
\begin{equation}
\int_E\nabla\phi\,d\mathcal{L}^n=-\int_{\Rn}\phi\,dD\mathbbm{1}_E.
\end{equation}
We denote by $P(E, \Omega)$ the perimeter of $E$ in $\Omega$ and  \[
P(E, \Omega)=\abs{D\mathbbm{1}_E}(\Omega),
\]
 where we denote by $\abs{D\mathbbm{1}_E}$ the total variation of the measure $D\mathbbm{1}_E$.
If $\Omega=\Rn$, we will write $P(E)$ instead of $P(E,\R^n)$.

As known, the topological notion of boundary is not relevant for sets with finite perimeter. Hence we recall the notions of reduced boundary and essential boundary. 
We start by the definition of densities.
 \begin{definition}Let $\Omega\subseteq\Rn$ be open, $S\in\mathcal{B}(\Omega)$ and $x\in\Omega$.
   \begin{enumerate}[\rm(i)]
        \item The upper $k$-dimensional density of $S$ at $x$ is defined by
       \begin{equation}
         \Theta^\ast_k(S,x)=\limsup_{\rho\rightarrow0^+}\frac{\mathcal{H}^k(S\cap B_{\rho}(x))}{\omega_k\rho^k}. \nonumber
       \end{equation}
     \item The lower $k$-dimensional density of $S$ at $x$ is defined by
       \begin{equation}
         \Theta_{\ast k}(S,x)=\liminf_{\rho\rightarrow0^+}\frac{\mathcal{H}^k(S\cap B_{\rho}(x))}{\omega_k\rho^k}. \nonumber
       \end{equation}
   \end{enumerate}
   If $\Theta^\ast_k(S,x)=\Theta_{\ast k}(S,x)$ we denote this common value with $\Theta_k(S,x)$ and we call it $k$-dimensional density of $S$ at $x$.
 \end{definition}

Let $\Omega\subseteq\Rn$ and let $\mu$ be a $\Rn$-valued measure on $(\Omega,\mathcal{B}(\Omega))$. The support of $\mu$ is the closed set denoted by ${\rm supp}(\mu)$ of all point s $x\in\Omega$ such that $\abs{\mu}(U)>0$ for all neighborhood $U$ of $x$. Moreover we say that $\mu$ is concentrated on $S\in\mathcal{B}(\Omega)$ if $\abs{\mu}(\Omega\sm S)=0$.

\begin{definition}
Let $\Omega\subseteq\Rn$ be open and let $E\in\mathcal{B}(\Rn)$.
\begin{enumerate}[\rm(i)]
\item The essential boundary $\partial^*E$ of $E$ is the set 
\[
\partial^*E=\Rn\sm(E^0\cup E^1),
\]
where 
\[
E^i=\Big\{x\in\Rn \ \big|\ \Theta_n(S,x)=i\Big\},\ \ i\in[0,1].
\]
\item If in addition $E$ has finite perimeter, then we define the reduced boundary $\mathcal{F}E$ of E  as the set of all $x\in {\rm supp}(D\mathbbm{1}_E)\cap\Omega$ such that 
\begin{equation}\label{gen_out_vect}
         \nu_E(x):=-\lim_{\rho\rightarrow0^+}\frac{D\mathbbm{1}_E(B_{\rho}(x))}{\abs{D\mathbbm{1}_E}(B_{\rho}(x))}, 
       \end{equation}
  exists and $\norm{\nu_E}=1$. We call the mapping $\nu_E:\mathcal{F}E\longrightarrow\Sn$ generalized outer normal to $E$.
\end{enumerate}
\end{definition}

Clearly $\mathcal{F}E$ is a Borel set of $\Omega$ and $\nu_E$ is a Borel mapping. Moreover by the Besicovitch derivation Theorem $\abs{D\mathbbm{1}_E}$ is concentrated on $\mathcal{F}E$ and $D\mathbbm{1}_E=\nu_E\abs{D\mathbbm{1}_E}$.
We state in the following theorem the main results due to De Giorgi and Federer on the structure of sets of finite perimeter. We recall first that a family of Borel sets $\{E_\rho\}_{\rho>0}\subseteq\Rn$ locally converges in measure in $\Rn$ to a set $E\subseteq\Rn$ if and only if $\mathbbm{1}_{E_\rho}$ converges to $\mathbbm{1}_E$ in $L^1_{\rm loc}(\Rn)$ as $\rho\searrow 0^+$.

\begin{theorem}\textup{\cite[Theorem 3.59 and Theorem 3.61]{amb_fus_pal_2000}}
Let $\Omega\subseteq\Rn$ be open. If $E\subseteq\Rn$ has finite perimeter in $\Omega$ then $\abs{D\mathbbm{1}_E}$ coincides with $\mathcal{H}^{n-1}\mr \mathcal{F}E$ and $P(E,\Omega)=\mathcal{H}^{n-1}(\mathcal{F}E$). Let. $x_0\in\mathcal{F}E$ the following statements hold.
\begin{enumerate}[\rm (i)]
\item the sets $\frac{E-x_0}{\rho}$ locally converge in measure in $\Rn$ as $\rho\searrow0^+$ to the halfspace
\[
H:=\Big\{x\in\Rn\ \big|\  \ps{\nu_E(x_0)}{x}\leq 0\Big\},
\]
where $\nu_E(x_0)$ is the generalized outer vector at $x_0$ and it is given by \eqref{gen_out_vect}.
\item ${\rm Tan}^{n-1}(\abs{D\mathbbm{1}_E},x_0)=\mathcal{H}^{n-1}\mr\nu_E^{\perp}(x_0)$ and 
\begin{equation}
\lim_{\rho\searrow0^+}\frac{\mathcal{H}^{n-1}\left(\mathcal{F}E\cap B_\rho(x_0)\right)}{\omega_{n-1}\rho^{n-1}}=1
\end{equation}
\end{enumerate}
Moreover $\mathcal{F}E\subseteq E^{\frac{1}{2}}\subseteq \partial^*E$ and $\mathcal{H}^{n-1}(\Omega\sm(E^0\cup\mathcal{F}E\cup E^1))=0$. In particular $E$ has density either $0$, $\frac{1}{2}$ or $1$ at $\mathcal{H}^{n-1}$-a.e. $x\in \Omega$ and 
\[
\mathcal{H}^{n-1}(\partial^*E\smallsetminus\mathcal{F}E)=0.
\]
\end{theorem}

Finally we need some notions on higher order rectifiability. The main reference is \cite{anz_ser_1994}.

Let  $\Lambda_k$ be the set of the $k$-vectors in $\Rn$ and by
$\Sigma_*$ the set of all non-zero simple $k$-vectors in $\Rn$. 
By $T(\tau)$ we denote the enveloping subspace of any $\tau\in \Sigma_*$.
 In the special case $s=2$, we have the following useful criteria for $C^2$--rectifiability. 
 
\begin{theorem}\textup{\cite[Corollary\,3.2.]{del_2003}}\label{delladio}
Let $E\subseteq\Rn$ be $\mathcal{H}^k$-rectifiable and $\pi: \Rn \times \Lambda_k\Rn\longrightarrow \Rn$ be the orthogonal projection on the first component. Then the following assertions are equivalent. 
\begin{enumerate}[\rm(i)]
\item The set $E$ is $\mathcal{H}^k$-rectifiable of class $C^2$.
\item There exists a $\mathcal{H}^k$-rectifiable set $F \subseteq \Rn \times \Lambda_k\Rn$ such that
\begin{enumerate}[\rm (a)]
\item $\pi(F)=E$;
\item $\tau\in\Sigma_*$ and $T(\tau)={\rm Tan}^k(\mathcal{H}^k\mr F,x)$, for $\mathcal{H}^k\mr F$-almost all $(x,\tau).$
\end{enumerate}
\end{enumerate}
\end{theorem}

\begin{remark}\label{graficolocale}
Let $\Sigma\subseteq\Rn$ be a $C^2$-hypersurface with principal curvatures $k_1(x),...,k_{n-1}(x)$. For all $\rho>0$ and $x\in\Sigma$ we denote by $Q^n_\rho(x)$ the $n$-dimensional cube with side $\rho$ and we set $Q^{n-1}_\rho(x):=Q^n_\rho(x)\cap(x+\nu_E(x)^\perp)$. 
Assume that $\Sigma$ is locally the graph of a $C^2$ function, namely for all $x\in\Sigma$  there exist $\rho>0$ such that $E\cap Q_\rho(x)$ is the sub-graph of a $C^2$ mapping $u: Q^{n-1}_\rho(x)\longrightarrow\R$ with $u(x)=0$ and $\nabla u(x)=0$. Then $D^2u(x)$ is a symmetric matrice and its eigenvalues coincides with the principal curvatures of $\Sigma$, namely there exists an orthonalmal basis $\{v_j(x)\ :\ j=1,...,n-1\}$ such that 
\begin{align*}
\ps{D^2u(x)v_i}{v_j}=\delta_{i,j}k_i(x).
\end{align*}
\end{remark}

\subsection{Sets with positive reach}
In this section we recall the main properties of sets with positive reach. 
This class of sets was first introduced by Federer in \cite{fed_1959}
in connection with the description of the curvature measures	. 
For more and recent results on curvature measure and on the 
structure on sets with positive reach we refer to 
\cite{rat_zah_2019,rat_zaj_2017,rot_zah_1990,zah_1986}. All the 
statement given in this section were originally proven in
\cite{fed_1959} but most of the results are also contained in the  
recent book by Rataj and Z{\"a}hle \cite{rat_zah_2019}.

Let $E\subseteq\Rn$: we define the set ${\rm Unp}(E)$ of all points 
having a unique projection on $E$ as 
\[
{\rm Unp}(E):=\Big\{y\in \Rn |\ \exists!\ x\in E\  \mbox{s.t.}\  
\delta_E(y)=\norm{x-y}\Big\}.
\]
The definition of ${\rm Unp}(E)$ implies that it is well defined a projection map  
\[
\pi_E:{\rm Unp}(E)\longrightarrow \ E
\] 
such that $\pi_E(y)=x,$ where $y$ is the unique element of $E$ such that $\delta_E(y)=\norm{x-y}$.   
 It can be shown that $\pi_E$ is continuous (see e.g. \cite[Lemma 4.1]{rat_zah_2019}).
We say that $E$ has positive reach at $x\in E$ if 
\[
{\rm reach}(E,x):=\sup\Big\{r>0\ \big|\ B_r(x)\subseteq {\rm Unp}(E)
\Big\}>0
\] 
and the reach of $E$ as 
\[
{\rm reach}(E)=\inf_{x\in E} {\rm reach}(E,x).
\] 

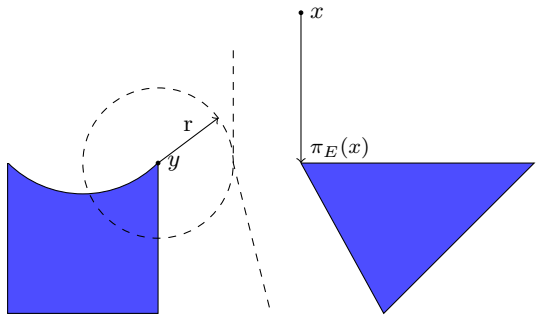
\begin{figure}[!htbp]
  \centering
   \mbox{%
     \begin{minipage}[l]{.25\textwidth}
      \caption{{\footnotesize In this picture the set $E$ is represented by the union of the two blue regions of the plane.  Moreover $r$ is the  ${\rm reach}(E,y)$ and the dashed line does not belong to ${\rm Unp}(E)$.} }
    \end{minipage}%
    \qquad
     \begin{minipage}[r]{0.5\textwidth}
      \begin{tikzpicture}[font=\small]
      \clip (-0.2,-2.1) rectangle + (7.3,7.2);
    \fill[blue!70] (0,0) rectangle (2,2);
    \draw (0,0) rectangle (2,2) ;
    \fill[white] (1,3) circle [radius=1.41];
    \draw (1,3) circle [radius=1.41];
    \fill[white] (-3,2) rectangle (3,5);
    \draw[dashed] (2,2) circle [radius=1];
    \filldraw (2,2) circle [radius=0.7pt] node[below,right] {$y$};
    \draw[->] (2,2)--(2.8,2.6) node[midway,above] {r};
    \draw[dashed] (3,3.5)--(3.,2)--(3.5,0);
    \filldraw (3.9,4) circle [radius=0.7pt] node[right] {$x$};
    \draw[->] (3.9,4)--(3.9,2);
    \fill[blue!70] (3.9,2)--(5,0)--(7,2)--(3.9,2);
    \draw (3.9,2)--(5,0)--(7,2)--(3.9,2);
    \node[above,right] at (3.9,2.2) {{\footnotesize$\pi_{E}(x)$}};
  \end{tikzpicture}
    \end{minipage}%
    }
\end{figure}

\begin{rmk}
We notice that 
\[
{\rm reach}\left(\frac{E-x_0}{\varrho}\right)
=\frac{1}{\varrho}{\rm reach}(E),\ \ \forall\rho>0.
\]
\end{rmk}

Clearly, if ${\rm reach}(E)>0$ then $E$ is closed and $\overline{E}\subseteq{\rm Unp}(E)$.
We stress that it is not required any kind of regularity of the boundary for sets with positive reach and hence standard definitions of tangent and normal spaces are not relevant in this setting. 

\begin{proposition}\label{prop_fed59}Let $E\subseteq\mathbb{R}^n$ be closed and $x\notin E$.
    \begin{enumerate}[\rm (i)]
      \item If $\delta_E$ is differentiable at $x$ then $x\in{\rm Unp}(E)$ and
      \begin{equation}
        \nabla\delta_E(x)=\frac{x-\pi_E(x)}{\delta_E(x)}.
      \end{equation}
      \item If $x\in{\rm Unp}(E)$ then $\delta_E$ is differentiable at $x$ 
	and then 
      \begin{equation}\label{147}
        x=\pi_E(x)+\delta_E(x)\nabla\delta_E(x),
      \end{equation}
      and $\nabla\delta_E(x)$ is the unique unit vector that satisfies \eqref{147}.
               \end{enumerate}
    \end{proposition}
    \begin{proof}
    (i) is proven in \cite[Theorem 4.8 (3)]{fed_1959} . (ii) can be found in \cite[Lemma 2.1.29]{hor_2007} and \cite[Lemma 8.5.12]{hor_2003}.
    \end{proof}

 \begin{definition}\label{coninormalietangenti} Let $E\subseteq\Rn$ and $x\in E$.
    \begin{enumerate} [\rm(i)]
      \item If $x$ is an isolated point of $E$ then ${\rm Tan}(E,x):=\{0\}.$ If $x$ is an accumulation point of $E$, we define as the tangent cone to $E$ at $x$ as the set 
        \begin{align}\nonumber
          {\rm Tan}(E,x):=&\   \biggl\{\lambda v\ \bigg|\ \lambda \geq0\ \mbox{and}\ \ \exists \ \{x_h\}_{h\in\mathbb{N}}\subseteq E\ \ \mbox{s.t.}\ \ x_h\longrightarrow x,\ x_h\neq x
          \\ &\mbox{and}\ v=\lim_{h\rightarrow+\infty}\frac{x_h-x}{\norm{x_h-x}}\ \biggl\}.\nonumber
                  \end{align}

      \item The normal cone to $E$ at $x$ is the dual cone of ${\rm Tan}(E,x)$, namely
        \begin{equation}
         {\rm Nor}(E,x):={\rm Dual}({\rm Tan}(E,x)):=\Big\{v\in\Rn\ |\ \ps{u}{v}\leq0\ \forall u\in{\rm Tan}(E,x)\Big\}. \nonumber
        \end{equation}
    \item          If $x$ is an isolated point of $E$ then ${\rm Tan}^{\musFlat}(E,x):=\{0\}.$ If $x$ is an accumulation point of $E$, we define as the adjacent cone to $E$ at $x$ as the set 
        \begin{align}\nonumber
          {\rm Tan}^{\musFlat}(E,x):=&\biggl\{v\in\Rn\ \bigg|\ \forall\ \{\lambda_h\}_{h\in\N}\subseteq(0,+\infty),\ \lambda_h\to 0^+\\
         & \exists\ \{x_h\}_{h\in\mathbb{N}}\subseteq E,\ \ \mbox{s.t.}\ \ x_h\longrightarrow x\ \ \mbox{and}\ \ v=\lim_{h\rightarrow+\infty}\frac{x_h-x}{\lambda_h}\biggl\}. \nonumber
        \end{align}

    \end{enumerate}   
  \end{definition}
Without any assumption on the boundary of $E$, many definition of tangent cone may be state. We refer to \cite[Chap. 4]{aub_fra_2009} and its bibliography for these topics. In particular it can be shown that
\begin{align*}
&{\rm Tan}(E,x)=\left\{v\in\Rn\ \bigg|\ \liminf_{s\rightarrow0^+}\frac{\delta_E(x+sv)}{s}=0\right\},\\
&{\rm Tan}^{\musFlat}(E,x)=\left\{v\in\Rn\ \bigg|\ \lim_{s\rightarrow0^+}\frac{\delta_E(x+sv)}{s}=0\right\}.
\end{align*}
Unfortunately, ${\rm Tan}(E,x)$ may fail to be convex if we have no extra hypothesis on $E$. Hence in general ${\rm Tan}(E,x)$ is a proper subset of ${\rm Dual}({\rm Nor}(E,x))$. However sets with positive reach enjoys more regularity property and we have the following statements.
     \begin{proposition}\label{P2} Let $E\subseteq\Rn$ be closed and non-empty. Let $x\in E$ and let $0<r<{\rm reach}(E,x)$. 
   \begin{enumerate}[\rm(i)]
      \item 
          ${\rm Nor}(E,x)=\Big\{\lambda v\in\Rn\ \big|\ \lambda\geq0,\ \norm{v}=r,\ \pi_E(x+v)=x\Big\}$.
       
      \item 
          ${\rm Tan}(E,x)={\rm Dual}({\rm Nor}(E,x))$.
       
      \item
           $\displaystyle{\lim_{s\rightarrow0^+}\frac{\delta_E(x+su)}{s}=0,\ \ \mbox{for any}\ u\in{\rm Tan}(E,x)}$.
           \item ${\rm Tan}(E,x)={\rm Tan}^{\musFlat}(E,x)$.
     
     \end{enumerate}
  \end{proposition}  
  \begin{proof}
  (i), (ii) and (iii) are proven in \cite[Theorem 4.8. (12)]{fed_1959}. (iv) follows by (iii).
  \end{proof}  
  The following lemmas list some crucial geometric properties of sets with positive reach.
\begin{lemma}\textup{\cite[Lemma 4.5]{rat_zah_2019}}\label{zahle2}
  Let $E\subseteq\Rn$, $x\in E$ and $R={\rm reach}(E,x)>0$. Let $v\in\Sn$, then the following statements are equivalent.
  \begin{enumerate}[\rm(i)]
    \item There exists $\displaystyle{s\in(0,R)}$ such that $\pi_E(x+sv)=x$.
    \item $\pi_E(x+sv)=x$ for all $\displaystyle{s\in(0,R)}$.
    \item $\displaystyle{E\cap B_{R}(x+Rv)=\emptyset}$.
    \item $v\in{\rm Nor}(E,x)$.
     \end{enumerate}  

 \end{lemma}
  
 \begin{lemma}\textup{\cite[Corollary 4.12]{rat_zah_2019}}\label{cor4.12} Let $E\subseteq\Rn$ be a set with positive reach.
 \begin{enumerate}[\rm (i)]
 \item If $x\in\partial E$ then ${\rm Nor}(E,x)\neq \{0\}$.
 \item If $u\in\Sn$ belongs to $\mathring{\overbrace{{\rm Tan}(E,x_0)}}$ then for some $\varepsilon>0$ the segment $[x,x+\varepsilon u]$ is included in $E$.
 \end{enumerate}
 \end{lemma}
 
  \begin{figure}[htbp]
  \centering
  \mbox{%
    \begin{minipage}[l]{.4\textwidth}
      \begin{tikzpicture}[font=\small]
    \fill[red!70] (0,2)--(0,0)--(4,0)--(4,1)..controls(3.5,1.2) and (0.12,0.1)..(0,2);
\draw (0,2)--(0,0)--(4,0)--(4,1)..controls(3.5,1.2) and (0.12,0.1)..(0,2) node[pos=-0.3, above] {$E$};

\draw (1,2) circle [radius=1];
\draw[->] (0,2)--(1,2) node[pos=0,left] {$x$};

\node[right] at (1,3.3) {$B_R(x+R\nu)$};

  \end{tikzpicture}
    \end{minipage}%
    \qquad
    \begin{minipage}[r]{.4\textwidth}
      \caption{{\small In this picture we show the behaviour of a set with positive reach $E$ as shown in Lemma \ref{cor4.12}(iii).}}
    \end{minipage}
   }
\end{figure}
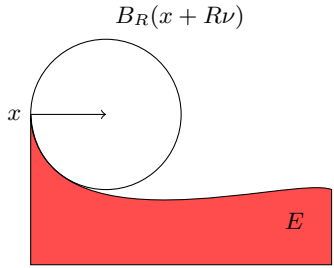

We state now the main results on curvature measures of a set with positive reach that were introduced by Federer in \cite{fed_1959}. Here we follow the approach of Rataj and Z\"ahle developed in \cite{zah_1986,rat_zah_2019}.

Let $E\subseteq\Rn$. The unit normal bundle of $E$ is the set
\[
 {\rm nor}^\star(E)=\{(x,\nu)\in\mathbb{R}^n\times\mathbb{S}^{n-1} \ \mbox{s.t.}\ x\in\partial E\ \mbox{and}\ \nu\in{\rm Nor}(E,x)\}.
 \]
Moreover we say that $(x,\nu)\in{\rm nor}^\star(E)$ is regular if $\nabla\delta_E$ is differentiable at $x+r\nu.$

Clearly ${\rm nor}^\star(E)$ is closed and ${\rm nor}^\star(E)$ is countably $\mathcal{H}^{n-1}$-rectifiable (see \cite[Corollary 4.22]{rat_zah_2019}). As proved in \cite[Proposition 4.23]{rat_zah_2019},
if ${\rm reach}(E)>0$ and $(x,\nu)\in{\rm nor}^\star(E)$ be a regular point, the generalized principal curvature of $E$ at $(x,\nu)$ are well defined $\mathcal{H}^{n-1}$-almost every $(x,\nu)\in{\rm nor}^\star(E)$ and denoted by $k_1(x,\nu),...,k_{n-1}(x,\nu)$. Moreover, such curvatures belong to $\left[-\frac{1}{{\rm reach}(E)},+\infty\right]$ and the functions
  \begin{align}\label{sigma0}
  \sigma_E^0(x,\nu)&=\frac{1}{\displaystyle{\prod_{i=1}^{n-1}\sqrt{1+k_i^2(x,\nu)}}},\\ 
  \sigma_E^k(x,\nu)&=\frac{\displaystyle{\sum_{1\leq i_1<...<i_k\leq n-1}\ \prod_{j=1}^k k_{i_j}(x,\nu)}}{\displaystyle{\prod_{i=1}^{n-1}\sqrt{1+k_i^2(x,\nu)}}},\ \ k\in\{1,...,n-1\},\label{sigmak}
  \end{align}
   are called generalized symmetric functions of the principal curvatures of order $k$ of $E$ at $(x,\nu)$. Since some of the generalized principal curvature could be equal to $+\infty$ we set these algebraic rules: 
\[
   \frac{1}{\sqrt{1+\infty^2}}=0,\ \ \frac{+\infty}{\sqrt{1+\infty^2}}=1.\]

We state here one of the main tools of this theory. By \cite[Theorem 4.30]{rat_zah_2019} and the standard approximation with simple functions we have the following result.

\begin{theorem} \textup{\textbf{(Global Steiner formula)}}\label{gsf} Let $E\subseteq\Rn$ be bounded and $0<r<{\rm reach}(E)$ then  for every non-negative measurable function $h:[E]_r\smallsetminus E\longrightarrow\R$
\begin{equation} \label{change}
\int_{[E]_r\sm E} h(z)\,d z=\sum_{k=0}^{n-1}\int_{{\rm nor}^\star(E)}\int_0^r h(y+s\nu)\,\sigma^k_E(x,\nu)s^k\,ds\,d\mathcal{H}^{n-1}(x,\nu).
\end{equation}
In particular the functions $\sigma^k_E$ are locally $\mathcal{H}^{n-1}$-integrable (and hence $\mathcal{H}^{n-1}$-measusrable) for all $k\in\{0,...,n-1\}$.
Moreover
\begin{equation} \label{steiner}
\mathcal{L}^n([E]_r\sm  E)=\sum_{k=0}^{n-1}\frac{r^{n-k}}{n-k}\int_{{\rm nor}^\star(E)}\sigma_E^{n-1-k}(x,\nu)d\,\mathcal{H}^{n-1}(x,\nu).
\end{equation}
\end{theorem}
Hence the curvature-direction measures of order $k\in\{0,...,n-1\}$ of $E$ are well defined  and they are Radon measures 
$\widetilde{C}_k(E,\cdot)$ on $\Rn\times\Sn$ given by 
\[
\widetilde{C}_k(E,B)=\frac{1}{(n-k)\omega_{n-k}}\int_{{\rm nor}^\star(E)\cap B}\sigma_E^{n-1-k}(x,\nu)d\,\mathcal{H}^{n-1}(x,\nu),\quad B\in\mathcal{B}(\Rn\times\Sn).
\]
The curvature measure $C_k(E,\cdot)$ defined by Federer in \cite[Sect. 5]{fed_1959} are given by 
\[
C_k(E,B)=\widetilde{C}_k(E,B\times\Sn),
\] 
for all $B\in\mathcal{B}(\Rn)$ for all $k\in\{0,...,n-1\}$.
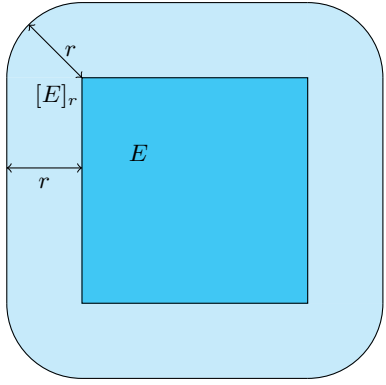
\begin{figure}[htbp]
  \centering
  \mbox{%
    \begin{minipage}{.4\textwidth}
\begin{tikzpicture}[scale=0.5, font=\small]
\filldraw[cyan! 20! white ] (-2,0) arc[start angle=180, end angle=90, radius=2]--(0,0);
\filldraw[cyan! 20! white ] (6,2) arc[start angle=90, end angle=0, radius=2]--(6,0);
\filldraw[cyan! 20! white ] (8,-6) arc[start angle=0, end angle=-90, radius=2]--(6,-6);
\filldraw[cyan! 20! white ] (0,-8) arc[start angle=-90, end angle=-180, radius=2]--(0,-6);
\filldraw[cyan! 20! white ] (-2,0)--(-2,-6)--(0,-6)--(0,0);
\filldraw[cyan! 20! white ] (0,2)--(6,2)--(6,0)--(0,0);
\filldraw[cyan! 20! white ] (8,0)--(8,-6)--(6,-6)--(6,0);
\filldraw[cyan! 20! white ] (0,-8)--(6,-8)--(6,-6)--(0,-6);
\filldraw[cyan! 60! white ] (0,0)--(0,-6)--(6,-6)--(6,0)--(0,0);
\draw (0,0)--(0,-6)--(6,-6)--(6,0)--(0,0);
\draw (-2,0)--(-2,-6);
\draw (0,2)--(6,2);
\draw (8,0)--(8,-6);
\draw (0,-8)--(6,-8);
\draw (-2,0) arc[start angle=180, end angle=90, radius=2];
\draw (6,2) arc[start angle=90, end angle=0, radius=2];
\draw (8,-6) arc[start angle=0, end angle=-90, radius=2];
\draw (0,-8) arc[start angle=-90, end angle=-180, radius=2];
\node[right] at (1,-2) {$E$};
\node[right] at (-1.5,-0.5) { $[E]_{r}$};
\draw[<->] (-2,-2.4)--(0,-2.4) node[midway,below] {$r$};
\draw[<->] (0,0)--(-1.414,1.414)node[midway,right] {$r$};
\end{tikzpicture}
    \end{minipage}%
   \begin{minipage}[r]{.35\textwidth}
     \caption{A set with positive reach and its parallel set.}
        \end{minipage}
   }
\end{figure}

\begin{remark}\label{casoregolare}
If $E$ is a compact $C^2$-domain the generalized principal curvatures agree with the classical Gaussian curvatures. 
More precisely, we consider firstly the orthogonal projection map 
\[
\pi:{\rm nor}^*(E)\longrightarrow\partial E
\]
whose Jacobian is given by
\[
(x,\nu)\longmapsto\frac{1}{\displaystyle{\prod_{i=1}^{n-1}\sqrt{1+k_i(x,\nu)^2}}}.
\] 
Appling the Area formula (\cite[Theorem 2.91]{amb_fus_pal_2000}), for all $k\in\{0,...,n-2\}$ and $B\in\mathcal{B}(\Rn)$, we get 
\begin{align*}
C_k(E,B)&=\frac{1}{(n-k)\omega_{n-k}}\int_{{\rm nor}^\star(E)\cap (B\times\Sn)}\sigma_E^{n-1-k}(x,\nu)d\,\mathcal{H}^{n-1}(x,\nu)\\
&=\frac{1}{(n-k)\omega_{n-k}}\int_{\partial E\cap B}\sum_{1\leq i_1<...<i_k\leq n-1}\ \prod_{j=1}^k k_{i_j}(x,\nu_E(x))d\,\mathcal{H}^{n-1}(x)\\
C_{n-1}(E,B)&=\frac{1}{(n-k)\omega_{n-k}}\int_{{\rm nor}^\star(E)\cap (B\times\Sn)}\sigma_E^{0}(x,\nu)d\,\mathcal{H}^{n-1}(x,\nu)\\
&=\frac{1}{(n-k)\omega_{n-k}}\mathcal{H}^{n-1}\left(\partial E\cap B\right)
\end{align*}
where the curvatures  $k_j(x):=k_j(x,\nu_E(x))$ for $j=1,...,n-1$ agrees with the principal curvatures of the smooth case.
\end{remark}

\begin{remark}
\label{sigmaInL1}
We point out that since the curvature measure are finite, their
total variations are finite, i.e. setting $\mu_k=\tilde C_k(E,\cdot)$, 
\[
|\mu_k|(\Rn\times \Sn)=
\frac{1}{(n-k)\omega_{n-k}}
\int_{{\rm nor}^\star(E)}
|\sigma_E^{n-1-k}(x,\nu)|
d\,\mathcal{H}^{n-1}(x,\nu)<+\infty,
\]
whence the fact that $\sigma_E^{n-1-k}\in L^1(\Rn\times \Sn,
\mathcal{H}^{n-1}\mr{\rm nor}^*(E))$.
\end{remark}

We end this section recalling basic results on sets with positive reach and the Hausdorff distance. We refer to \cite{amb_til_2004} for more details on the Hausdorff distance on a metric space and to \cite{fed_1959} for the properties of sets with positive reach.
   \begin{definition} Let $(X,d)$ be a metric space. 
    \begin{enumerate}[{\rm(i)}]
    \item By $\mathcal{C}_X$ we denote the set of all non-empty 
    closed subsets of $X$. 
    \item For $A, B\in\mathcal{C}_X$ we define the  
    Hausdorff distance as 
    \[
    \textup{dist}_\textup{H}(A,B)=\inf\Big\{t>0\ \big|\ 
    B\subseteq [A]_t\ \mbox{and}\ A\subseteq [B]_t\Big\}.
    \]
    \end{enumerate}
   \end{definition}
    \begin{definition} Let $(X,d)$ be a metric space. Let $C$ be a closed subset of $X$ and let $\{C_j\}_{j\in\N}\subseteq\mathcal{C}_X$. We write $C_j\xrightarrow{H}C$ if $\{C_j\}_{j\in\N}$ converges to $C$ with respect to the topology induced by the Hausdorff distance.
   \end{definition}
  \begin{definition} Let $(X,d)$ be a metric space. Let $C$ be a closed subset of $X$ and let $\{C_j\}_{j\in\N}\subseteq\mathcal{C}_X$. We say that the sequence $\{C_j\}_{j\in\N}$ converges in the sense of Kuratowski to $C$ if the following conditions are satisfied.
  \begin{enumerate}[{\rm (i)}]
    \item If $\displaystyle{x=\lim_{j\rightarrow+\infty}x_{j}}$ where $\{x_{j}\}_{j\in\N}$ is a sequence with $x_j\in C_j$ for all $j\in\N$, then $x\in C$.
    \item If $x\in C$ then there exists $\{x_j\}_{j\in\N}$ such that $x_j\in C_j$ for all $j\in\N$ such that $\displaystyle{x=\lim_{j\rightarrow+\infty}x_{j}}$.
  \end{enumerate}
  In this case we write $C_j\xrightarrow{K}C.$
   \end{definition}
     \begin{remark}\label{osshaus} We recall some known facts.
     \begin{enumerate}[{\rm (i)}]
     \item  $(\mathcal{C}_X,\textup{dist}_\textup{H})$ is a complete metric space and if $X$ is compact then $(\mathcal{C}_X,\textup{dist}_\textup{H})$ is compact (see \cite[Proposition 4.4.12, Proposition 4.4.14]{amb_til_2004}). 
     \item If $C_j\xrightarrow{H}C$ then $C_j\xrightarrow{K}C$ and if $X$ is compact the converse is also true. For a proof look at \cite[Theorem 4.4.15]{amb_til_2004}.
     \end{enumerate}
     \end{remark}

 \begin{proposition}\textup{\cite[Remark 4.14]{fed_1959}}\label{chiusocp} Let $\varepsilon>0$ and $K\subseteq\mathbb{R}^n$ be compact.
 \begin{enumerate}[{\rm(i)}]
   \item $B^{\varepsilon}=\Bigl\{ \emptyset\neq A\subseteq\mathbb{R}^n\ |\ {\rm reach}(A)\geq \varepsilon\Bigr\}$ is closed with rispect to the Hausdorff metric.
   \item $C^{\varepsilon}=\Bigl\{ \emptyset\neq A\subseteq K\ |\ {\rm reach}(A)\geq \varepsilon\Bigr\}$ is compact with rispect to the Hausdorff metric.
 \end{enumerate}
 \end{proposition}

 \begin{theorem}\textup{\cite[Theorem 5.9 Remark 5.10]{fed_1959}}\label{convdeb}
 Let $\{C_j\}_{j\in\N}$ be a sequence of compact subset of $\Rn$ such that:
\begin{enumerate}[{\rm(i)}]
\item there exists $\varepsilon>0$ such that 
\[
\inf_{j\in\N}{\rm reach}(C_j)\geq \varepsilon,
\]
\item there exists a compact $C\in\Rn$ such that $C_j\xrightarrow{H} C$. 
\end{enumerate}
Then for all $k\in\{0,...,n-1\}$ the curvature measure $C_k(C_j,\cdot)$ weakly converges to $C_k(C,\cdot)$ as $j\to+\infty$, namely for all $f\in C_b(\Rn)$ and $B\in\mathcal{B}(\Rn)$ we have
\[
\lim_{j\to +\infty}
\int_{B} f(y)\, dC_k(C_j,dy)=\int_{B} f(y)\, dC_k(C,dy).
\]
\end{theorem}

\subsection{Singular points}

For a set with positive reach $E$ is relevant to split $E$ (and in particular $\partial E$) as the union of suitable Borel set whose definition is based on the  dimension of the normal cone. This kind of decomposition was introduced by Federer in \cite[Remark 4.15 and Remark 4.20]{fed_1959} and it was further studied by Hug in the context of the convex geometry (see \cite[Lemma 3.1 and Theorem 3.2]{hug_1998})  and by Kohlmann in a Riemmanian setting (see \cite{koh_1991}).
      
We consider the orthogonal projections maps $\pi_1:\Rn\times\Sn\longrightarrow\Rn$ and $\pi_2:\Rn\times\Sn\longrightarrow\Sn$ such that for all $x\in\Rn$ and $y\in\Sn$ we have 
\[
\pi_1(x,y)=x,\ \ \pi_2(x,y)=y.
\] 

\begin{definition}Let $E\subseteq\Rn$ be a set with positive reach and $k\in\{0,...,n\}$.
\begin{enumerate}[\rm (i)]
 \item The set of the $k$-singular points of $E$ is the set 
\begin{equation}
E_{k}=\Big\{x\in E\ \big|\  {\rm dim}({\rm Nor}(E,x))\geq n-k
\Big\}.
\end{equation}
\item The set of the classical ridge points of order $k$ is the set
\begin{align}
\Sigma_{0}=&E_{0}, \\
\Sigma_k=&E_k\sm E_{k-1}=\Big\{x\in E\ \big|\  {\rm dim}({\rm Nor}(E,x))= n-k
\Big\},\ \ k\in\{1,...,n\}.
\end{align}
\item The set of the $k$-singular points of the unit normal bundle of $E$ is the set
\begin{equation}
E^k=\pi_1^{-1}(E_k)\cap{\rm nor}^\star(E),\ \ k\in\{0,...,n\}.
\end{equation} 
\item The set of the $k$-ridge points of the unit normal bundle of $E$ is the set
\begin{align}
\Sigma^k=&\pi_1^{-1}(\Sigma_k)\cap{\rm nor}^\star(E),\ \ k\in\{0,...,n\}.
\end{align}
\end{enumerate}
\end{definition}

\begin{lemma}\label{decoE} Let $E\subseteq\Rn$ be a set with positive reach and $k\in\{0,...,n\}$. 
\begin{enumerate}[\rm (i)]
\item $E_k$ and $\Sigma_k$ are countably $k$-rectifiable Borel sets.
\item  $E^k$ and $\Sigma^k$ are countably $\mathcal{H}^{n-1}$-rectifiable Borel sets. 
\end{enumerate}
Moreover for $k\in\{0,...,n-1\}$, the following statements hold true.
\begin{enumerate}[\rm (iv)]
\item[\rm (iv)]For $\mathcal{H}^{n-1}$a.e. $z\in E^k$ at least $n-1-k$ of the generalized curvatures $k_1(z),...,k_{n-1}(z)$ are equal to $+\infty$.
\item[\rm (v)]  For $\mathcal{H}^{n-1}$a.e. $z\in \Sigma^k$ exactly  $n-1-k$ of the generalized curvatures $k_1(z),...,k_{n-1}(z)$ are equal to $+\infty$.
\item[\rm (vi)] $\sigma^{n-1-k}$ is non-zero on $E^{k}$ and $\widetilde{C}_k$ is concentrated on $E^k$.

\end{enumerate}
\end{lemma}

   \begin{proof}
  (i) The rectifiability of $E_k$ is proven in \cite[Lemma 3.1]{hug_1998} and in \cite[Remark 4.15 (3)]{fed_1959}. The rectifiability of
  $\Sigma_k$ follows from that of $E_k$. 
  Since  ${\rm nor}^\star(E)$ is countably $\mathcal{H}^{n-1}$-rectifiable (see \cite[Corollary 4.22]{rat_zah_2019}) then (iii) follows. (iv) is proven in \cite[Corollary 2.11]{koh_1991} and (iv) implies (v) and (vi).
     \end{proof}
     
By \cite[Theorem 3.2]{hug_1998} and the standard approximation with simple functions we have the following Coarea formula for the curvature measures.

\begin{theorem}\label{hug}
Let $E\subseteq\Rn$ be a set with positive reach and $B\subseteq\Rn\times\Sn$ be a bounded Borel set. Then for $k\in\{0,...,n-1\}$ and for every bounded Borel function $h:{\rm nor}^\star(E)\longrightarrow\R$
{\small\begin{align}
\int_{E^k\cap B}h(x,\nu)\sigma^{n-1-k}_E(x,\nu)\,d\mathcal{H}^{n-1}(x,\nu)&=\int_{E_k}\int_{{\rm Nor}(E,x)\cap \pi_2(B)}h(x,\nu)\,d\mathcal{H}^{n-1-k}(\nu)\, d\mathcal{H}^k(x)
 \nonumber \\&=\int_{\Sigma_k}\int_{{\rm Nor}(E,x)\cap \pi_2(B)}h(x,\nu)\,d\mathcal{H}^{n-1-k}(\nu)\, d\mathcal{H}^k(x).
\end{align}}
\end{theorem}

 \section[Positive reach vs finite perimeter]{Sets with positive reach as sets with finite perimeter}\label{ps_rettificabilità}

In this section we investigate features of sets with positive reach as particular sets with finite perimeter. The main reference is \cite{amb_col_vil_2008} and we point out some higher order rectifiability results.

 We set 
 \[
 \Sigma_{n-1}^i=
 \Big\{x\in\Sigma_{n-1}\ \big|\
 \mathcal{H}^0({\rm Nor}(E,x)\cap\Sn)=i
 \Big\},\ \ i\in\{1,2\}.
 \]

\begin{proposition}\label{summingup}
Let $E\subseteq\Rn$ be a compact set with positive reach. The following statements hold true.
\begin{enumerate}[\rm (i)]
\item $E$ has finite perimeter, $\partial E$ is countably $\mathcal{H}^{n-1}$-rectifiable and $P(E)\leq\mathcal{H}^{n-1}(\partial E)<+\infty$.
\item $\mathcal{H}^0({\rm Nor}(E,x)\cap\Sn)\in\{1,2\}$ for $\mathcal{H}^{n-1}$-almost every $x\in\partial E$. Moreover $P(E)=\mathcal{H}^{n-1}(\partial E)$ if and only if $\mathcal{H}^0({\rm Nor}(E,x)\cap\Sn)=1$ for a.e. $x\in\partial E$.
\item  $\Sigma_{n-1}^1\subseteq \partial^*E$ and $\mathcal{H}^{n-1}(\partial^*E\sm\Sigma_{n-1}^1)\leq \mathcal{H}^{n-1}(\partial^*E\sm E^{\frac{1}{2}})=0$.
\item If $B$ is a closed ball of radius $\rho$, $E\cap B\neq\emptyset$ and ${\rm reach}(E,x)>\rho$ for all $x\in E\cap B$, then ${\rm reach}(E\cap B)>\rho$.
\item $\partial E$ is $\mathcal{H}^{n-1}$-rectifiable of class $C^2$. 
\end{enumerate}
\end{proposition}
\begin{proof}
(i), (ii) are proved in \cite[Proposition 3, Theorem 9]{amb_col_vil_2008}. (iii) is essentially contained in the proof of \cite[Theorem 9]{amb_col_vil_2008} and (iv) is proved in \cite[Lemma 3.4 (i)]{rat_zaj_2017}. Since  ${\rm nor}^\star(E)$ is countably $\mathcal{H}^{n-1}$-rectifiable (see \cite[Corollary 4.22]{rat_zah_2019}), (v) is implied by Theorem \ref{delladio} with $R=\partial E$, $\pi=\pi_1$ and $F={\rm nor}^\star(E)$. 
\end{proof}

Now we want to characterize the blow-ups at any point of the boundary of a set with positive reach $E$. So we set $\displaystyle{E_{x_0,\rho}=\frac{E-x_0}{\rho}}$ for all $x_0\in\partial E$ and for all $\rho>0$. We start by this technical lemma that is a generalization of \cite[Theorem pag. 64]{beer_1974}.

\begin{lemma}\label{beer} Let $\{C_j\}_{j\in\N}$ be a sequence of compact subsets of $\Rn$ such that:
\begin{enumerate}[{\rm(i)}]
\item there exists $\varepsilon>0$ such that 
\[
\inf_{j\in\N}{\rm reach}(C_j)\geq \varepsilon,
\]
\item there exists a compact $C\in\Rn$ such that $C_j\xrightarrow{H} C$. 
\end{enumerate}
Then $\{C_j\}_{j\in\N}$  converges in measure to $C$.
\end{lemma}

\begin{proof}
By \cite[Theorem 4.13]{fed_1959}, ${\rm reach}(C)\geq \varepsilon$. Let $\{\rho_h\}_{h\in\N}\subseteq (0,\varepsilon)$ such that $\rho_h \searrow 0^+$. For all $h\in\N$, using the definition given in  
\eqref{annulus}, we can consider the family 
\[
B_h=\Big\{A_{\rho_h}(C_j)\ \big|\ j\in \N\Big\}
\cup\Big\{A_{\rho_h}(C)\Big\}.
\] 
By Theorem \ref{convdeb} and \eqref{steiner}  for all $h\in\N$ we have \[
A_{\rho_h}(C_j)\xrightarrow{H}A_{\rho_h}(C),\ \ \mbox{as}\ j\to+\infty.
\] 
Hence for all $h\in\N$ the family $B_h$ is compact with respect to the Hausdorff metric.
Since the Lebesgue measure is upper semicontinuous with respect to  
Hausdorff convergence, by  Weierstrass Theorem for all $h\in\N$ 
there exist $F_h\in B_h$ with maximal Lebesgue measure. Since ${\rm reach}(F_h)\geq \varepsilon-\rho_1$ for all $h\in\N$ the family $\{F_h\}_{h\in\N}$ is compact with respect to the Hausdorff metric and there exists 
a converging subsequence of $\{F_h\}_{h\in\N}$ whose limit is either the boundary of $C$ or the boundary of $C_j$ for some $jp\in\N$. Hence $\displaystyle{\lim_{h\to+\infty}\mathcal{L}^n(F_h)=0}$. Hence
\begin{equation}
0\leq \sup_{j\in\N}\abs{\mathcal{L}^n\left([C_j]_{\rho_h}\right)-\mathcal{L}^n(C_j)}\leq\mathcal{L}^n(F_h),
\end{equation}
i.e. $\mathcal{L}^n\left([C_j]_{\rho_h}\right)$ converges  
to $\mathcal{L}^n\left(C_j\right)$ uniformly
in $j\in \N$, and then 
by \cite[Theorem 1]{beer_1974} we obtain the statement. 
\end{proof}

\begin{theorem}\label{convhaus} Let $E\in\R^n$ be a set with positive reach.  If $x_0\in\partial E$ then for all $r>0$ the sets 
$E_{x_0,\rho}\cap \overline{B_r(0)}$ converge with respect to the Hausdorff distance as $\rho\to0^+$ to  ${\rm Tan}(E,x_0)\cap \overline{B_r(0)}$.
\end{theorem}

\begin{proof}
Let $r>0$ and let $\rho\in (0,1)$. Since
\[
{\rm reach}\left(E_{x_0,\rho}\right)=
\frac{1}{\rho}{\rm reach}(E)>{\rm reach}(E),
\] 
by Proposition \ref{summingup}(iv) there exists $\rho_0>0$ such that  \[
{\rm reach}\left(E_{x_0,\rho}\cap \overline{B_r(0)}\right)>r,\quad 
\forall\ \rho\in(0,\rho_0).
\] 
Then the family $\{E_{x_0,\rho}\cap 
\overline{B_r(0)}\}_{\rho\in(0,\rho_0)}$ is compact with respect to the 
Hausdorff  metric by Corollary \ref{chiusocp}(ii) and there exists a 
sequence $\{\rho_h\}_{h\in\N}$ and a compact set 
$F\subseteq \overline{B_r(0)}$ such that $\rho_h\to 0^+$ and 
$E_{\rho_h}\cap B_r(0)\to F$ with respect the Hausdorff distance as 
$h\to+\infty$. 

We prove now that $F={\rm Tan}^{\musFlat}(E,x_0)\cap \overline{B_r(0)}$. Let $y\in F$ then $y\in\overline{B_r(0)} $ and by Remark \ref{osshaus}(ii)  there exists  $\{x_h\}_{h\in\mathbb{N}}\subseteq E$ such that $x_h\longrightarrow x_0$ and $\displaystyle{y=\lim_{h\rightarrow+\infty}\frac{x_h-x_0}{\rho_h}}$. Hence $y\in{\rm Tan}^{\musFlat}(E,x_0)\cap \overline{B_r(0)}$. Conversely, let $y\in{\rm Tan}^{\musFlat}(E,x_0)\cap \overline{B_r(0)}$ then for all $\{r_h\}_{h\in\N}\subseteq(0,+\infty)$ such that $r_h\to 0^+$ there exists $\{x_h\}_{h\in\mathbb{N}}\subseteq E$ such thar $x_h\longrightarrow x_0$ and $\displaystyle{ v=\lim_{h\rightarrow+\infty}\frac{x_h-x}{r_h}}$. Choosing $r_h=\rho_h$ for all $h\in\N$ we have $y\in F$.

By Proposition \ref{P2}(iii) $F={\rm Tan}(E,x_0)\cap \overline{B_r(0)}$ and since $F$ does not depend on the choice of the subsequence $\{\rho_h\}_{h\in\N}$, the statement follows.

\end{proof}

\begin{corollary} \label{blowup} Let $E\in\R^n$ be a set with positive reach.  If $x_0\in\partial E$ then the sets 
$E_{x_0,\rho}$ locally converges in measure as $\rho\to 0^+$ to ${\rm Tan}(E,x_0)$.  
\end{corollary} 

\begin{proof}
Apply Theorem \ref{convhaus} and Proposition \ref{beer}.
\end{proof}
  
\section{Taylor expansion of the heat content of a set with positive reach}\label{taylor}

In this section we compute the second order Taylor expansion of the heat content of a set with positive reach.

\begin{lemma}
 Let $E\subseteq\Rn$ be a compact set with positive reach and let $\varphi\in C^1_b(\R^n)$. Then for all $\gamma>0$ we have 
 \begin{align}
 \frac{1}{(4\pi)^{\frac{n}{2}}t^n}\int_{E^c\sm[E]_r} \int_E e^{-\frac{\norm{z-y}^2}{4t^2}}\varphi(y)\,dy\,dz=o(t^\gamma),\ \ r\in\bigl(0,{\rm reach}(E)\bigr), 
 \end{align}
 and
 {\footnotesize\begin{align}\label{formulabella}
f_{E}(t)=\frac{1}{(4\pi)^{\frac{n}{2}}}\sum_{k=0}^{n-1}t^{k+1}\Biggl(\int_0^{\frac{r}{t}}\,d\beta\int_{{\rm nor}^\star(E)}\,d\mathcal{H}^{n-1}(x,\nu)\int_{\frac{E-x}{t}}\beta^k\sigma^k_E(x,\nu)e^{-\frac{\norm{w-\beta\nu}^2}{4}}\varphi(x+tw)\,dw\Biggr)+o(t^n).
\end{align}}
\end{lemma}
\begin{proof}
Let $0<r<{\rm reach}(E)$. Then 
\begin{align*}
f_{E}(t)=\frac{1}{(4\pi)^{\frac{n}{2}}t^n}&\int_{E^c}\int_{E} e^{-\frac{\norm{z-y}^2}{4t^2}}\varphi(y)\,dy\,dz=\frac{1}{(4\pi)^{\frac{n}{2}}t^n}\biggl(\int_{[E]_r\sm E}\int_E e^{-\frac{\norm{z-y}^2}{4t^2}}\varphi(y)\,dy\,dz\nonumber \\
+&\int_{E^c\sm[E]_r}\int_E e^{-\frac{\norm{z-y}^2}{4t^2}}\varphi(y)\,dy\,dz\biggr).
\end{align*}

If $z\in E^c\sm[E]_r$, then $\norm{z-y}\geq r$ for all $y\in E$ and
\begin{align*}
-\frac{\norm{z-y}^2}{4t^2}\leq-\frac{\norm{z-y}^2}{8t^2}-\frac{r^2}{8t^2}
\end{align*}
Then we get
\begin{align*}
\abs{\frac{1}{(4\pi)^{\frac{n}{2}}t^n}\int_{E^c\sm[E]_r} \int_E e^{-\frac{\norm{z-y}^2}{4t^2}}\varphi(y)\,dy\,dz}&\leq\frac{\norm{\varphi}_\infty}{(4\pi)^{\frac{n}{2}}t^n}e^{-\frac{r^2}{8t^2}}\int_{E^c\sm[E]_r} \int_E e^{-\frac{\norm{z-y}^2}{8t^2}}\,dy\,dz\nonumber \\
&\leq \frac{2^{\frac{n}{2}}\norm{\varphi}_\infty}{\pi^{\frac{n}{2}}} e^{-\frac{r^2}{8t^2}} \mathcal{L}^n(E)\int_{\R^n}e^{-\norm{w}^2}dw\nonumber\\
&\leq 2^{\frac{n}{2}}\norm{\varphi}_\infty e^{-\frac{r^2}{8t^2}}\mathcal{L}^n(E)
=o(t^\gamma), \qquad \forall \gamma>0,
\end{align*}
where $\displaystyle{w=\frac{z-y}{2\sqrt{2}t}}$ and $dz= 8^{\frac{n}{2}}t^ndw$.
In particular, 
applying \eqref{change} and we obtain
{\small\begin{align*}
&f_{E}(t)=\frac{1}{(4\pi)^{\frac{n}{2}}t^n}\sum_{k=0}^{n-1}\int_0^rds\int_{{\rm nor}^\star(E)}d\mathcal{H}^{n-1}(x,\nu)\int_E \sigma^k_E(x,\nu)\,s^ke^{-\frac{\norm{x+s\nu-y}^2}{4t^2}}\varphi(y)\,dy+o(t^n)=\nonumber \\
=&\frac{1}{(4\pi)^{\frac{n}{2}}}\sum_{k=0}^{n-1}\left(t^{k+1}\int_0^{\frac{r}{t}}d\beta\int_{{\rm nor}^\star(E)}d\mathcal{H}^{n-1}(x,\nu)\int_{\frac{E-x}{t}}\sigma^k_E(x,\nu)\,\beta^ke^{-\frac{\norm{w-\beta\nu}^2}{4}}\varphi(x+tw)\,dw\right)+o(t^n),
\end{align*}}
where $\displaystyle{w=\frac{y-x}{t}}$, $\displaystyle{\beta=\frac{s}{t}}$, $dy= t^n dw$ and $ds= t\,d\beta$.
\end{proof}
We prove now the main theorem of this paper.

\begin{theorem}\label{teorema_sviluppo_ps}
Let $E\subseteq\Rn$ be a compact set with positive reach and let $\varphi\in C^1_b(\R^n)$. Then
\begin{align}\label{sviluppops}
f_E(t)&=\frac{t}{\sqrt{\pi}}\int_{\mathcal{F}E}\varphi(x)d\mathcal{H}^{n-1}(x)
+
t^2\biggl[\int_{\mathcal{F}E}\left(\alpha_nH_E(x)\varphi(x)-\ps{\nabla\varphi(x)}{\nu_E(x)}\right)\,d\mathcal{H}^{n-1}(x)\nonumber\\
&+\int_{\Sigma_{n-2}} c_{n-2}(x)\varphi(x)\,d\mathcal{H}^{n-2}(x)\biggr]+o(t^2),
\end{align}
where 
\begin{align}
&\alpha_n=2^{n-2} \pi^{\frac{n-1}{2}}(n-1),\\
&H_E(x)=\frac{1}{n-1}\sum_{j=1}^{n-1}k_j(x),\\
&\Sigma_{n-2}=\Bigl\{x\in E\ \big|\  {\rm dim}({\rm Nor}(E,x))=2\Bigr\},\\
&c_{n-2}(x)=\frac{1}{(4\pi)^{\frac{n}{2}}}\int_{{\rm Nor}(E,x)}d\mathcal{H}^{2}(y)\int_{{\rm Tan}(E,x)}\norm{y}\,e^{-\frac{\norm{w-y}^2}{4}}\,dw,\ \ x\in\Sigma_{n-2},
\end{align}
and the curvatures $k_i(x)$ are meant as in Remark \ref{casoregolare} since $\mathcal{F}E$ is $\mathcal{H}^{n-1}$-rectifiable of class $C^2$.

In particular, taking $\varphi\equiv 1$ we get 
{\small\begin{align}\label{sviluppoK}
K_{t^2}(E,E^c)&=\frac{t}{\sqrt{\pi}}P(E)+t^2\biggl[\alpha_n\int_{\mathcal{F}E}H_E(x)\,d\mathcal{H}^{n-1}(x)
+\int_{\Sigma_{n-2}} c_{n-2}(x)\,d\mathcal{H}^{n-2}(x)\biggr]+o(t^2),
\end{align}}
and
{\footnotesize\begin{align} \label{normaheat2}
&\norm{T_t\mathbbm{1}_E}_{L^2(\R^n)}=\sqrt{\frac{2t}{\pi}}P(E)+2t\biggl[\alpha_n\int_{\mathcal{F}E}H_E(x)
\,d\mathcal{H}^{n-1}(x)
+\int_{\Sigma_{n-2}} c_{n-2}(x)\,d\mathcal{H}^{n-2}(x)\biggr]+o(t),\\
&\norm{T_t\mathbbm{1}_E-\mathbbm{1}_E}_{L^1(\R^n)}=
2\sqrt{\frac{t}{\pi}}P(E)+2t\biggl[\alpha_n\int_{\mathcal{F}E}H_E(x)
\,d\mathcal{H}^{n-1}(x)
+\int_{\Sigma_{n-2}} c_{n-2}(x)\,d\mathcal{H}^{n-2}(x)\biggr]+o(t).\label{normaheat1}
\end{align}}
\end{theorem}
\begin{proof}
Given $t>0$, by \eqref{formulabella} and Lemma \ref{decoE}(vi) we have
\begin{align*} 
\frac{f_E(t)}{t}=\frac{1}{(4\pi)^{\frac{n}{2}}}\int_0^{\frac{r}{t}}d\beta\int_{E^{n-1}}d\mathcal{H}^{n-1}(x,\nu)\int_{\frac{E-x}{t}}\sigma^0_E(x,\nu)e^{-\frac{\norm{w-\beta\nu}^2}{4}}\varphi(x+tw)\,dw+o(1).
\end{align*}
Since 
\begin{align*}
\abs{\sigma^0_E(x,\nu) e^{-\frac{\norm{w-\beta\nu}^2}{4}}\varphi(x+tw)}\leq \norm{\varphi}_\infty \abs{\sigma^0_E(x,\nu) }e^{-\frac{\norm{w-\beta\nu}^2}{4}};
\end{align*}
thanks to Remark~\ref{sigmaInL1} and  the Dominated Convergence Theorem we get
\begin{align}\label{persecondostep}
\lim_{t\to0^+}\frac{f_E(t)}{t}&=\frac{1}{(4\pi)^{\frac{n}{2}}}\int_0^{+\infty}d\beta\int_{E^{n-1}}d\mathcal{H}^{n-1}(x,\nu)\int_{{\rm Tan}(E,x)}\sigma^0_E(x,\nu)e^{-\frac{\norm{w-\beta\nu}^2}{4}}\varphi(x)\,dw.
\end{align}
 For every $x\in\mathcal{F}E$ we denote by $\nu_E(x)$ the generalized outer normal at $x$ to $E$ given by \eqref{gen_out_vect} and 
 \begin{equation}\label{normale}
 {\rm Nor}(E,x)=\{t\nu_E(x),\ t\geq 0\}.
 \end{equation}
  Hence 
  \begin{equation}\label{tangente}
  {\rm Tan}(E,x)=\left\{w'-\alpha\nu_E(x)\ |\ \alpha\geq 0,\ w'\in\nu_E(x)^\perp\right\}
  \end{equation}
   and by Theorem \ref{hug} and Proposition \ref{summingup}(iii) we get
 \begin{align}
&\lim_{t\to0^+}\frac{f_E(t)}{t}=\frac{1}{(4\pi)^{\frac{n}{2}}}\int_0^{+\infty}d\beta\int_{\mathcal{F}E}d\mathcal{H}^{n-1}(x)\int_{{\rm Nor}(E,x)\cap\Sn}d\mathcal{H}^{0}(\nu)\int_{{\rm Tan}(E,x)}e^{-\frac{\norm{w-\beta\nu}^2}{4}}\varphi(x)\,dw\nonumber\\
&=\frac{1}{(4\pi)^{\frac{n}{2}}}\int_0^{+\infty}d\beta\int_{\mathcal{F}E}d\mathcal{H}^{n-1}(x)\int_{{\rm Nor}(E,x)\cap\Sn}d\mathcal{H}^{0}(\nu)\int_{\nu_E(x)^\perp}d\alpha\int_0^{+\infty}e^{-\frac{\norm{w'-(\alpha+\beta)\nu_E(x)}^2}{4}}\varphi(x)\,dw'\nonumber\\
&=\frac{1}{(4\pi)^{\frac{n}{2}}}\int_0^{+\infty}\int_0^{+\infty}e^{-\frac{(\alpha+\beta)^2}{4}}\,d\alpha\,d\beta\int_{\R^{n-1}}e^{-\frac{\norm{w'}^2}{4}}\,dw'\int_{\mathcal{F}E}\varphi(x)d\mathcal{H}^{n-1}(x).\label{calcoliutili}
\end{align}
By standard computations we have
\begin{align}
&\int_0^{+\infty}\int_0^{+\infty}e^{-\frac{(\alpha+\beta)^2}{4}}\,d\alpha\,d\beta=2,\label{int1}\\
&\int_{\R^{n-1}}e^{-\frac{\norm{w'}^2}{4}}\,dw'=(4\pi)^{\frac{n-1}{2}},\label{int2}
\end{align}
and 
\begin{align}\label{firstorder}
\lim_{t\to0^+}\frac{f_E(t)}{t}=\frac{1}{\sqrt{\pi}}\int_{\mathcal{F}E}\varphi(x)d\mathcal{H}^{n-1}(x)=:f'_E(0^+).
\end{align}
We compute now
\begin{align*}
\frac{f_E(t)-tf'_E(0^+)}{t^2}=&\frac{1}{(4\pi)^{\frac{n}{2}}}\int_0^{\frac{r}{t}}d\beta\int_{E^{n-2}}d\mathcal{H}^{n-1}(x,\nu)\int_{\frac{E-x}{t}}\beta\sigma^1_E(x,\nu)e^{-\frac{\norm{w-\beta\nu}^2}{4}}\varphi(x+tw)\,dw\\
&+\frac{1}{t}\biggl[\frac{1}{(4\pi)^{\frac{n}{2}}}\int_0^{\frac{r}{t}}d\beta\int_{E^{n-1}}d\mathcal{H}^{n-1}(x,\nu)\int_{\frac{E-x}{t}}\sigma^0_E(x,\nu)\,e^{-\frac{\norm{w-\beta\nu}^2}{4}}\varphi(x+tw)\,dw\\
&-\frac{1}{\sqrt{\pi}}\int_{\mathcal{F}E}\varphi(x)d\mathcal{H}^{n-1}(x)\biggr]+o(1)=:p(t)+\frac{q(t)}{t}+o(1).
\end{align*}

Since 
\begin{align*}
\abs{\beta\sigma^1_E(x,\nu) e^{-\frac{\norm{w-\beta\nu}^2}{4}}\varphi(x+tw)}\leq \norm{\varphi}_\infty \abs{\sigma^1_E(x,\nu) }\beta e^{-\frac{\norm{w-\beta\nu}^2}{4}};
\end{align*}
thanks to Remark~\ref{sigmaInL1}, we can use the 
Dominated Convergence Theorem to obtain
\begin{align*}
\lim_{t\to0^+}p(t)&=\frac{1}{(4\pi)^{\frac{n}{2}}}\int_0^{+\infty}d\beta\int_{E^{n-2}}d\mathcal{H}^{n-1}(x,\nu)\int_{{\rm Tan}(E,x)}\beta\sigma^1_E(x,\nu)\,e^{-\frac{\norm{w-\beta\nu}^2}{4}}\varphi(x)\,dw\\
&=\frac{1}{(4\pi)^{\frac{n}{2}}}\int_0^{+\infty}d\beta\int_{\Sigma_{n-2}}d\mathcal{H}^{n-2}(x)\int_{{\rm Nor}(E,x)\cap\Sn}d\mathcal{H}^{1}(\nu)\int_{{\rm Tan}(E,x)}\beta\,e^{-\frac{\norm{w-\beta\nu}^2}{4}}\varphi(x)\,dw,
\end{align*}
where we have exploited Theorem \ref{hug} and Proposition \ref{summingup}(iii).
Setting $y=\beta\nu$, by to the Coarea Formula we get

\begin{align}\label{limitep}
\lim_{t\to0^+}p(t)&=\frac{1}{(4\pi)^{\frac{n}{2}}}\int_{\Sigma_{n-2}}d\mathcal{H}^{n-2}(x)\int_{{\rm Nor}(E,x)}d\mathcal{H}^{2}(y)\int_{{\rm Tan}(E,x)}\norm{y}\,e^{-\frac{\norm{w-y}^2}{4}}\varphi(x)\,dw\\
&=\int_{\Sigma_{n-2}}\varphi(x) c_{n-2}(x)d\mathcal{H}^{n-2}(x).
\end{align}

Since 
\[
\varphi(x+tw)=\varphi(x)+t\ps{\nabla\varphi(x)}{w}+o(t),
\] 
we have
\begin{align*}
q(t)&=\frac{1}{(4\pi)^{\frac{n}{2}}}\int_0^{\frac{r}{t}}\int_{E^{n-1}}\sigma^0_E(x,\nu)\,\int_{\frac{E-x}{t}}e^{-\frac{\norm{w-\beta\nu}^2}{4}}\varphi(x+tw)\,dw\,d\mathcal{H}^{n-1}(x,\nu)\,d\beta-f'_E(0^+)\\
&=\frac{1}{(4\pi)^{\frac{n}{2}}}\int_0^{\frac{r}{t}}\int_{E^{n-1}}\sigma^0_E(x,\nu)\int_{\frac{E-x}{t}}e^{-\frac{\norm{w-\beta\nu}^2}{4}}\varphi(x)\,dw\,d\mathcal{H}^{n-1}(x,\nu)\,d\beta\\
&+\frac{t}{(4\pi)^{\frac{n}{2}}}\int_0^{\frac{r}{t}}\int_{E^{n-1}}\sigma^0_E(x,\nu)\,\int_{\frac{E-x}{t}}e^{-\frac{\norm{w-\beta\nu}^2}{4}}\ps{\nabla\varphi(x)}{w}\,dw\,d\mathcal{H}^{n-1}(x,\nu)\,d\beta-f'_E(0^+)+o(t).
\end{align*}
We set
\begin{align*}
q_1(t):=\frac{1}{(4\pi)^{\frac{n}{2}}}\int_0^{\frac{r}{t}}\int_{E^{n-1}}\sigma^0_E(x,\nu)\int_{\frac{E-x}{t}}e^{-\frac{\norm{w-\beta\nu}^2}{4}}\varphi(x)\,dw\,d\mathcal{H}^{n-1}(x,\nu)\,d\beta-f'_E(0^+),\\
q_2(t):=\frac{t}{(4\pi)^{\frac{n}{2}}}\int_0^{\frac{r}{t}}\int_{E^{n-1}}\sigma^0_E(x,\nu)\,\int_{\frac{E-x}{t}}e^{-\frac{\norm{w-\beta\nu}^2}{4}}\ps{\nabla\varphi(x)}{w}\,dw\,d\mathcal{H}^{n-1}(x,\nu)\,d\beta.
\end{align*}
We compute now the Taylor expansion of $g_1$. 
Since 
\[
\int_{\frac{r}{t}}^{+\infty}\int_0^{+\infty}
e^{-\frac{(\alpha+\beta)^2}{4}}\,d\alpha\,d\beta=
o(t^\gamma),\quad \mbox{for all}\ t^\gamma,
\] 
\eqref{calcoliutili} implies that for all $\gamma>0$ we have
\begin{equation}\label{primoordineapprossimato}
f'_E(0^+)=\frac{1}{(4\pi)^{\frac{n}{2}}}\int_0^{\frac{r}{t}}\int_{E^{n-1}}\sigma_E^0(x,\nu)\int_{{\rm Tan}(E,x)}e^{-\frac{\norm{w-\beta\nu}^2}{4}}\varphi(x)\,dw\,d\mathcal{H}^{n-1}(x,\nu)\,d\beta+o(t^\gamma)
\end{equation}
and 
\begin{align}\label{g1bis}
q_1(t)=\frac{1}{(4\pi)^{\frac{n}{2}}}\int_0^{\frac{r}{t}}\int_{\mathcal{F}E}F(\beta,x)\varphi(x)\,d\mathcal{H}^{n-1}(x)\,d\beta+ o(t),
\end{align}
where for all $\ x\in\mathcal{F}E$ and $0\leq\beta\leq\frac{r}{t}$
\begin{align}\label{Fbetax}
F(\beta,x)=\int_{\frac{E-x}{t}\sm{\rm Tan}(E,x)}e^{-\frac{\norm{w-\beta\nu_E(x)}^2}{4}}\,dw-\int_{{\rm Tan}(E,x)\sm\frac{E-x}{t}}e^{-\frac{\norm{w-\beta\nu_E(x)}^2}{4}}\,dw.
\end{align}
Now for all $\rho>0$ and $x\in\mathcal{F}(E)$ we denote by $Q^n_\rho(x)$ the $n$-dimensional cube with side $\rho$ and we set $Q^{n-1}_\rho(x):=Q^n_\rho(x)\cap(x+\nu_E(x)^\perp)$. By Proposition \ref{summingup}(v) $\mathcal{F}E$ is locally the graph of a $C^2$-mapping (see Figure \ref{graficoC2}). So for every $x\in\mathcal{F}E$ there exist $\rho>0$ such that $E\cap Q_\rho(x)$ is the sub-graph of a $C^2$ mapping $u: Q^{n-1}_\rho(x)\longrightarrow\R$. Without loss of generality we can assume $u(x)=0$ and $\nabla u(x)=0$ and we have
\begin{align*}
E\cap Q_\rho(x)=\Big\{w'+h\nu_E(x)\ \bigl|\  -\rho\leq h\leq u(w'),\ w'\in Q^{n-1}_\rho(x)\Big\}.
\end{align*}

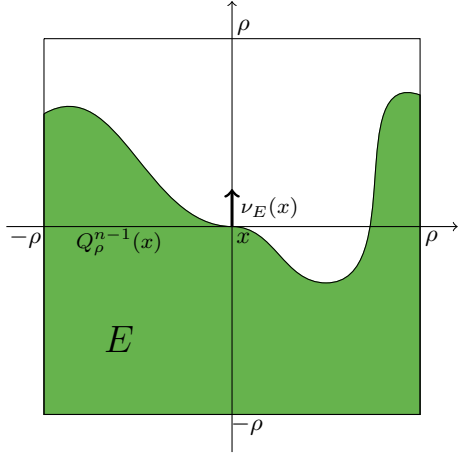
\begin{figure}[htbp]\label{graficoC2}
  \centering
  \mbox{%
    \begin{minipage}{.5\textwidth}
\begin{tikzpicture}[scale=0.5, font=\small]
 \tikzset{axes/.style={}}
\begin{scope}[style=axes]

 \filldraw[teal! 60! yellow](-5,3) to [out=30,in=180]  (0,0) to [out=0, in=180] (2.5,-1.5)to [out=0, in=160] (5,3.5)to [out=-90, in=90] (5,-5)to [out=180, in=0] (-5,-5)to [out=90, in=-90] (-5,3); 
 \draw(-5,3) to [out=30,in=180]  (0,0) to [out=0, in=180] (2.5,-1.5)to [out=0, in=160] (5,3.5)to [out=-90, in=90] (5,-5)to [out=180, in=0] (-5,-5)to [out=90, in=-90] (-5,3); 
 \draw[->] (-6,0) -- (6,0); \node at (-3,-.4) {{\footnotesize$Q^{n-1}_\rho(x)$}};
\draw[->] (0,-6)-- (0,6); \node at (1,.5) {{\footnotesize$\nu_E(x)$}};
\draw[->, line width=0.4mm] (0,0)--(0,1);
\draw (-5,5)--(-5,-5)--(5,-5)--(5,5)--(-5,5);
  \node at (-5.5,-.3) {$-\rho$};
 \node at (.4,-5.3) {$-\rho$};
  \node at (5.3,-.3) {$\rho$};
\node at (.3,5.3) {$\rho$};
\node at (-3,-3) {{\Large $E$}};
\node at (.3,-.3) {$x$};
\end{scope}
\end{tikzpicture}

    \end{minipage}
       \begin{minipage}[r]{.35\textwidth}
      \caption{{\small In this picture we  represent how the boundary of $E$ can be seen locally as the graph of a $C^2$ mapping.}}
   \end{minipage}
   }
\end{figure}

Setting  $\displaystyle{u_{x,t}\left(\frac{w'-x}{t}\right)=\frac{u(w')}{t}}$, we have $u_{x,t}(0)=0$, $\nabla u_{x,t}(0)=0$ and 
\begin{align*}
\frac{E-x}{t}\cap Q_{\frac{\rho}{t}}(0)=\left\{\frac{w'+h\nu_E(x)-x}{t}\ \biggl|\  -\frac{\rho}{t}\leq \frac{h}{t}\leq u_{x,t}\left(\frac{w'-x}{t}\right),\ \frac{w'-x}{t}\in Q^{n-1}_{\frac{\rho}{t}}(0)\right\}.
\end{align*}
Hence setting $\displaystyle{z'=\frac{w'-x}{t}}$ and $\displaystyle{\alpha=\frac{h}{t}}$, we have $w=z'+\alpha\nu_E(x)$ and 
\begin{align*}
F(\beta,x)&=\int_{Q^{n-1}_{\frac{\rho}{t}}(0)}\int_0^{u_{x,t}(z')}e^{-\frac{\norm{z'+(\alpha-\beta)\nu_E(x)}^2}{4}}\,d\alpha dz'\\
&=\int_{Q^{n-1}_{\frac{\rho}{t}}(0)}e^{-\frac{\norm{z'}^2}{4}}\int_0^{u_{x,t}(z')}e^{-\frac{(\alpha-\beta)^2}{4}}\,d\alpha dz'. 
\end{align*}

We set now 
\[
G(t)=\int_0^{u_{x,t}(z')}e^{-\frac{(\alpha-\beta)^2}{4}}\,d\alpha
\]
and clearly
\begin{align*}
G'(t)=e^{-\frac{(u_{x,t}(z')-\beta)^2}{4}} \frac{d}{dt}(u_{x,t}(z')).
\end{align*}
By the Taylor Formula we have
\begin{align*}
u_{x,t}(z')= \frac{u(x+tz')}{t}=t\int_0^1s\,ds\int_0^1\ps{D^2u(x+\sigma st) z'}{z'}\,d\sigma
\end{align*}

Hence 
\begin{align*}
\lim_{t\to0^+}\frac{G(t)}{t}=\lim_{t\to0^+}G'(t)= e^{-\frac{\beta^2}{4}}\int_0^1s\,ds\int_0^1\ps{D^2u(x)z'}{z'}\,d\sigma=\frac{1}{2}\ps{D^2u(x)z'}{z'}e^{-\frac{\beta^2}{4}}.
\end{align*}

Since
\begin{align*}
\abs{\sigma^0_E(x,\nu) e^{-\frac{\norm{w-\beta\nu}^2}{4}}\varphi(x)}\leq \norm{\varphi}_\infty \abs{\sigma^0_E(x,\nu) }e^{-\frac{\norm{w-\beta\nu}^2}{4}},
\end{align*}
by Remark~\ref{sigmaInL1} and the Dominated convergence Theorem 
\begin{align*}
\lim_{t\to0^+}\frac{q_1(t)}{t}&=\frac{(4\pi)^{\frac{n-1}{2}}}{(4\pi)^{\frac{n}{2}}}\frac{1}{2}\int_{\mathcal{F}E}\int_{\R^{n-1}}e^{-\frac{\norm{z'}^2}{4}}\ps{D^2u(x)z'}{z'}\varphi(x)\,dz'd\mathcal{H}^{n-1}(x)\int_0^{+\infty}e^{-\frac{\beta^2}{4}}\,d\beta\\
&=\frac{1}{4}\int_{\mathcal{F}E}\int_{\R^{n-1}}e^{-\frac{\norm{z'}^2}{4}}\ps{D^2u(x)z'}{z'}\,\varphi(x)\,dz'd\mathcal{H}^{n-1}(x).
\end{align*}

By Remark \ref{graficolocale} and Remark \ref{casoregolare}, if we choose a orthonormal basis $\{v_j: j=1,...,n-1\}$ of eigenvectors of $D^2u(x)$ we have 
\begin{align}
&\ps{D^2u(x)z'}{z'}=\sum_{j=1}^{n-1}(z'_j)^2k_j(x) ,\\
&\int_{\R^{n-1}}(z'_j)^2 e^{-\frac{\norm{z'}^2}{4}}dz'= 2^n \pi^{\frac{n-1}{2}}.
\end{align}
Hence
\begin{align}\label{limiteg1}
\lim_{t\to0^+}\frac{q_1(t)}{t}&=2^{n-2} \pi^{\frac{n-1}{2}}(n-1)\int_{\mathcal{F}E}H_E(x)\varphi(x)\,d\mathcal{H}^{n-1}(x).
\end{align}

We compute now the Taylor expansion of $g_2$. Since
\begin{equation*}
\abs{\sigma^0_E(x,\nu)e^{-\frac{\norm{w-\beta\nu}^2}{4}}\ps{\nabla\varphi(x)}{w}}\leq \abs{\sigma^0_E(x,\nu)}\norm{\nabla\varphi}_\infty\norm{w}e^{-\frac{\norm{w-\beta\nu}^2}{4}},
\end{equation*}
by Remark~\ref{sigmaInL1} and
the Dominated convergence Theorem and Theorem \ref{hug} we have
\begin{align*}
&\lim_{t\to0^+}\frac{q_2(t)}{t}=\frac{1}{(4\pi)^{\frac{n}{2}}}\int_0^{+\infty}\int_{E^{n-1}}\sigma^0_E(x,\nu)\,\int_{{\rm Tan}(E,x)}e^{-\frac{\norm{w-\beta\nu}^2}{4}}\ps{\nabla\varphi(x)}{w}\,dw\,d\mathcal{H}^{n-1}(x,\nu)\,d\beta\\
&=\frac{1}{(4\pi)^{\frac{n}{2}}}\int_0^{+\infty}\int_{\mathcal{F}E}\int_{{\rm Nor}_{E}(x)\cap\Sn}\int_{{\rm Tan}(E,x)}e^{-\frac{\norm{w-\beta\nu}^2}{4}}\ps{\nabla\varphi(x)}{w}\,dw\,d\mathcal{H}^{0}(\nu)\,d\mathcal{H}^{n-1}(x)\,d\beta.
\end{align*}
Recalling \eqref{normale} and \eqref{tangente}, we get
\begin{align*}
&\frac{1}{(4\pi)^{\frac{n}{2}}}\int_0^{+\infty}\int_{\mathcal{F}E}\int_{{\rm Nor}_{E}(x)\cap\Sn}\int_{{\rm Tan}(E,x)}e^{-\frac{\norm{w-\beta\nu}^2}{4}}\ps{\nabla\varphi(x)}{w}\,dw\,d\mathcal{H}^{0}(\nu)\,d\mathcal{H}^{n-1}(x)\,d\beta\\
&=\frac{1}{(4\pi)^{\frac{n}{2}}}\int_0^{+\infty}\int_{\mathcal{F}E}\int_{{\rm Nor}(E,x)\cap\Sn}\int_{\R^{n-1}}\int_0^{+\infty}e^{-\frac{\norm{w'-(\alpha+\beta)\nu_E(x)}^2}{4}}\ps{\nabla\varphi(x)}{w'-\alpha\nu_E(x)}\\
&\,dw'\,d\alpha\,d\mathcal{H}^{0}(\nu)\,d\mathcal{H}^{n-1}(x)\,d\beta\\
&=\frac{1}{(4\pi)^{\frac{n}{2}}}\int_{\mathcal{F}E}\biggl[\int_0^{+\infty}\int_0^{+\infty}e^{-\frac{(\alpha+\beta)^2}{4}}d\alpha\,d\beta\ps{\nabla\varphi(x)}{\int_{\R^{n-1}}e^{-\frac{\norm{w'}^2}{4}}w' dw'}\\
&-\int_{\R^{n-1}}e^{-\frac{\norm{w'}^2}{4}}\,dw'\int_0^{+\infty}\int_0^{+\infty}\alpha e^{-\frac{(\alpha+\beta)^2}{4}}d\alpha\,d\beta\ps{\nabla\varphi(x)}{\nu_E(x)}\biggr]d\mathcal{H}^{n-1}(x).
\end{align*}
By standard computation we have
\begin{align}
&\int_{\R^{n-1}}e^{-\frac{\norm{w'}^2}{4}}w' dw'=0,\\
&\int_0^{+\infty}\int_0^{+\infty}\alpha e^{-\frac{(\alpha+\beta)^2}{4}}d\alpha\,d\beta=\sqrt{4\pi},
\end{align}
and recalling \eqref{int1} and \eqref{int2} we have
\begin{align}\label{limiteg2}
&\lim_{t\to0^+}\frac{q_2(t)}{t}=-\int_{\mathcal{F}E}\ps{\nabla\varphi(x)}{\nu_E(x)}d\mathcal{H}^{n-1}(x).
\end{align}
By \eqref{firstorder}, \eqref{limitep}, \eqref{limiteg1} and \eqref{limiteg2} the statement follows.

Finally, \eqref{sviluppoK} follows by \eqref{sviluppops} with $\varphi\equiv1$ and \eqref{normaheat2},\eqref{normaheat2} follow by \eqref{norma2}, \eqref{norma1} and \eqref{sviluppoK}.
\end{proof}

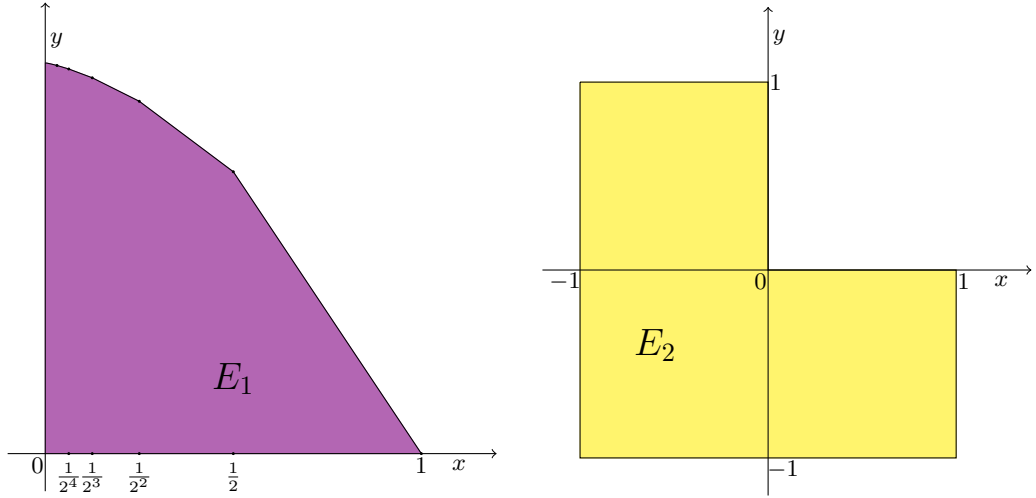
\begin{figure}[htbp]
  \centering
  \mbox{%
    \begin{minipage}{.45\textwidth}
\begin{tikzpicture}[scale=0.5, font=\small]
 \tikzset{axes/.style={}}
\begin{scope}[style=axes]

\filldraw[violet! 60! white](10,0) --(10/2,30*2^-2)--(10*2^-2,30*5/16)--(10*2^-3,30*1/3)--(10*2^-4,30*131/384)--(10*2^-5,1983/192)--(10*2^-6,3*1327/384)--(10*2^-7,3*1163/336)--(10*2^-8,3*3.464)--(0,3*3.464)--(0,0)--(10,0);
\draw(10,0) --(10/2,30*2^-2)--(10*2^-2,30*5/16)--(10*2^-3,30*1/3)--(10*2^-4,30*131/384)--(10*2^-5,1983/192)--(10*2^-6,3*1327/384)--(10*2^-7,3*1163/336)--(10*2^-8,3*3.464)--(0,3*3.464);
 \filldraw (10,0) circle [radius=0.7pt];
 \filldraw (10/2,30*2^-2) circle [radius=0.7pt];
  \filldraw (10*2^-2,30*5/16) circle [radius=0.7pt];
  \filldraw (10*2^-3,30*1/3) circle [radius=0.7pt];
  \filldraw (10*2^-4,30*131/384) circle [radius=0.7pt];
  \filldraw (10*2^-5,1983/192) circle [radius=0.7pt];
    \filldraw (5,0) circle [radius=0.7pt] node[below] {$\frac{1}{2}$};
  \filldraw (2.5,0) circle [radius=0.7pt] node[below] {$\frac{1}{2^2}$};
;
  \filldraw (5/4,0) circle [radius=0.7pt] node[below] {$\frac{1}{2^3}$};
;
  \filldraw (5/8,0) circle [radius=0.7pt] node[below] {$\frac{1}{2^4}$};
;

 \draw[->] (-1,0) -- (12,0); \node at (11,-.3) {$x$};
\draw[->] (0,-1)-- (0,12); \node at (.3,11) {$y$};
\node at (-.2,-.3) {$0$};
 \node at (10,-.3) {$1$};
   \node at (5,2) {{\Large$E_1$}};

\end{scope}
\end{tikzpicture}

    \end{minipage}
   \begin{minipage}[r]{.35\textwidth}
\begin{tikzpicture}[scale=0.5, font=\small]
 \tikzset{axes/.style={}}
 
\begin{scope}[style=axes]
\filldraw[yellow! 70! white] (-5,5)--(-5,-5)--(5,-5)--(5,0)--(0,0)--(0,5)--(-5,5);
\draw (-5,5)--(-5,-5)--(5,-5)--(5,0)--(0,0)--(0,5)--(-5,5);
 \draw[->] (-6,0) -- (7,0); 
\draw[->] (0,-6)-- (0,7); 
 \end{scope}
  \node at (-3,-2) {{\Large$E_2$}};
 \node at (5.2,-.3) {$1$};
  \node at (-.2,-.3) {$0$};
 \node at (.2,5) {$1$};
 \node at (-5.4,-.3) {$-1$};
 \node at (.4,-5.3) {$-1$};
 \node at (6.2,-.3) {$x$};
\node at (.3,6.2) {$y$};
\end{tikzpicture}
   \end{minipage}}
   \caption{{\small Representation of the sets $E_1$ and $E_2$.}}
\end{figure}

We end this section with some remarks.
\begin{rmk}\label{remark_finale} In \cite{ang_mas_mir_2013}, the authors obtained a third order Taylor expansion of $f_E(t)$ assuming that $E$ has skeleton $C^{1,1}$. We make a comparison with this previous result.
\begin{enumerate}[{\rm (i)}]
\item There exist sets with positive reach that do not have skeleton $C^{1,1}$. As an example in the set $E_1=\{(x,y)\in\R^2\ |\ x\in[0,1]\ \mbox{and}\ y\in[0,u(x)]\}$ where $u:[0,1]\longrightarrow\R$ is a continuous given by 

\[
u(x)=3\sum_{n=0}^{+\infty}u_n(x)\mathbbm{1}_{\left[\frac{1}{2^{n+1}},\frac{1}{2^n}\right]}(x)
\]
and $u_n$ are affine functions such that 
\begin{align*}
&u(1)=0,\\
&u\left(\frac{1}{2^n}\right)=\sum_{i=1}^n\frac{1}{i 2^{i+1}}.
\end{align*}
\item There exist sets with $C^{1,1}$ skeleton by with zero reach. Take for instance the set $E_2=[-1,0]\times[-1,1]\cup[0,1]\times[-1,0]$.
So the set $f_{E_2}(t)$ admits a third order Taylor expansion even if ${\rm reach}(E_2)=0$.
\item The term of \eqref{sviluppops} containing $c_{n-2}$ is an implicit version of the term $\Theta_1$ defined in \cite[Lemma 4.6]{ang_mas_mir_2013} where the tangent and the normal cones were explicitly given by the positive linear combination of explicit vectors.
\end{enumerate}
\end{rmk}


\end{document}